\documentclass[12pt]{article}
\usepackage{amsmath,amssymb,amsfonts,amsthm}
\usepackage{amscd}

\newcommand{\Z}{\mathbb{Z}}
\newcommand{\C}{\mathbb{C}}

\newcommand{\cO}{\mathcal{O}}
\newcommand{\cE}{\mathcal{E}}
\newcommand{\cL}{\mathcal{L}}

\newcommand{\bP}{\mathbb{P}}

\newcommand{\bE}{\mathsf{E}}
\newcommand{\bB}{\mathsf{B}}
\newcommand{\bM}{\mathsf{M}}
\newcommand{\bL}{\mathsf{L}}
\newcommand{\bN}{\mathsf{N}}
\newcommand{\bW}{\mathsf{W}}
\newcommand{\bZ}{\mathsf{Z}}

\newcommand{\cF}{\mathcal{F}}

\newcommand{\bA}{\mathsf{A}}
\newcommand{\cM}{\mathcal{M}}

\newcommand{\fm}{\mathfrak{m}}
\newcommand{\fM}{\mathfrak{M}}
\newcommand{\fP}{\mathfrak{P}}
\newcommand{\Lb}{\overline{L}}

\newcommand{\wt}{\widetilde}

\DeclareMathOperator{\slope}{slope}
\DeclareMathOperator{\Coker}{Coker}
\DeclareMathOperator{\Jac}{Jac}
\DeclareMathOperator{\supp}{supp}
\DeclareMathOperator{\Coh}{Coh}

\DeclareMathOperator{\Bl}{Bl}
\DeclareMathOperator{\rk}{rk}
\DeclareMathOperator{\Ker}{Ker}

\DeclareMathOperator{\Pic}{Pic}

\DeclareMathOperator{\Hilb}{\fM}
\DeclareMathOperator{\ParHilb}{\fP}
\DeclareMathOperator{\Mat}{Mat}
\DeclareMathOperator{\Aut}{Aut}

\DeclareMathOperator{\Tor}{Tor}
\DeclareMathOperator{\Ext}{Ext}
\DeclareMathOperator{\Hom}{Hom}
\DeclareMathOperator{\End}{End}
\DeclareMathOperator{\Mod}{Mod}
\DeclareMathOperator{\Bimod}{Bimod}
\DeclareMathOperator{\Tails}{Tails}
\DeclareMathOperator{\Spec}{Spec}
\DeclareMathOperator{\dv}{div}
\DeclareMathOperator{\tr}{tr}
\DeclareMathOperator{\Tr}{Tr}
\DeclareMathOperator{\Res}{Res}

\newtheorem{Proposition}{Proposition}
\newtheorem{Theorem}{Theorem}
\newtheorem{Lemma}{Lemma}

\usepackage[all]{xy}

\begin{document}
\title{Noncommutative geometry and Painlev\'e equations}
\author{Andrei Okounkov and Eric Rains}
\date{}
\maketitle

\begin{abstract}
We construct the elliptic Painlev\'e equation and its higher 
dimensional analogs as the action of line bundles on 1-dimensional 
sheaves on noncommutative surfaces. 
\end{abstract}

\section{Introduction}

\subsection{}

The classical Painlev\'e equations are very special 2-dimensional 
dynamical systems; they and their generalizations (including 
discretizations) appear in many applications. Their theory is 
very well developed, in fact, from many different angles, see 
for example \cite{Painleve100} for an introduction. 
Many of these approaches
are very geometric and some can be interpreted in terms of 
noncommutative geometry\footnote{
In particular, in \cite{Arinkin/Borodin}, Arinkin and Borodin gave an 
algebro-geometric interpretation of a degenerate discrete Painlev\'e
 equation. Their dynamics takes place not on the moduli spaces
of sheaves but rather on moduli of discrete analogs of connections. 
Their construction may, in fact, be 
interpreted in terms of ours, as will be shown in \cite{Rains2}}.
A full discussion of the relation
between the two topics in the title is outside of the scope of
the present note. 

Our goals here are very practical. The dynamical systems
we discuss appear in a very simple, yet challenging, problem 
of probability theory and mathematical physics: planar 
dimer (or lattice fermion) with a changing boundary, see 
\cite{Kenyon} for an introduction and \cite{Okounkov1,Okounkov2} for the developments
that lead to the present paper. The link to 
Artin-style noncommutative geometry, which is 
the subject of this paper, turns out to be very useful for 
dynamical and probabilistic applications. 

Our hope is to promote further interaction between the two fields and, 
with that goal in mind, we state most of our results in the minimal 
interesting generality, with only a hint of the bigger picture. 
We also emphasize explicit examples.

\subsection{}\label{s_gen}

In algebraic geometry, there is an abundance of group actions 
of the following kind. Let $S\subset\bP^N$ be a projective algebraic variety
(it will be a surface in what follows, whence the choice of notation)
and let 
$$
\bA = \C[x_0,\dots,x_N]\Big/(\textup{equations of $S$}) 
$$
be its homogeneous coordinate ring. Coherent sheaves $\cM$ on $S$ may 
be described as follows 
\begin{equation}
\Coh S =  \frac{\textup{finitely generated
graded $\bA$-modules $\bM$}}{\textup{those of finite dimension}} \,. 
\label{CohS} 
\end{equation}
They depend on discrete as well as
continuous parameters, so that 
$$
X = \textup{moduli space of $\cM$} 
$$
is a countable union of algebraic varieties. While there is a 
very developed general theory of such moduli spaces, see e.g.\ \cite{HL}, 
one can get a very concrete sense of $X$ by giving generators and 
relations for $\bM$, as we will do below. 

The group $\Pic(S)$ of line bundles $\cL$ on $S$ acts on $X$ by 
\begin{equation}
\cM \mapsto \cL \otimes \cM\,. 
\label{tens}
\end{equation}
If $\cL$ is topologically nontrivial, this permutes connected
components of $X$. A great many \emph{integrable} actions of 
abelian groups can be understood from this perspective, an 
obvious invariant of the dynamics being the cycle in $S$ given 
by the support of $\cM$. 

\subsection{}

Our point of departure in this paper is the observation that $\bA$ 
need not be commutative for the constructions of Section \ref{s_gen}.
In fact, noncommutative projective geometry in the 
sense of M.~Artin \cite{SvdB} is precisely the study of graded algebras with 
a good category \eqref{CohS}.

The key new feature of the noncommutative situation is that for 
tensor products
like \eqref{tens} one needs a \emph{right} $\bA$-module $\bL$ 
and then 
$$
\bL \otimes_{\bA} \bM \in \Mod(\bA')\,, \quad \bA' = \End_\bA(\bL) \,.
$$
If $\bL$ is a deformation of a line bundle then $\bA'$ is closely 
related to $\bA$ but, in general, 
$$
\bA' \not \cong \bA \,.
$$
as can be already seen in very simple examples, see Section 
\ref{nc_examples}. 
As a result, we have
\begin{equation}
X \xrightarrow{\,\, \cL \otimes  \,\,} X'
\label{XX'}
\end{equation}
where $X'$ is the corresponding moduli space for $\bA'$. 

While this sounds very abstract, we will be talking about a very 
concrete special case in which $S$ is a blow-up of another surface $S_0$ 
$$
S = \textup{Blow-up of $S_0$ at $p\in S_0$}
$$
and $\cL$ is the exceptional divisor. In the noncommutative case, 
tensoring with $\cL$ will make the point $p$ move in $S_0$ by 
an amount proportional 
to the strength of noncommutativity, see Section \ref{nc_examples}. 

\subsection{}

Noncommutativity deforms the dynamics in two ways. First, the 
action \eqref{XX'} 
happens on a larger space that parametrizes both the module $\bM$ 
and the algebra $\bA$, with an invariant fibration given by 
forgetting $\bM$. Specifically, we will be talking about sheaves
on blowups of $\bP^2$, where the centers $p_1,\dots,p_n$ 
of the blowup are allowed
to move on a fixed cubic curve $E\subset \bP^2$. There will be 
a $\Z^n$-action on these that covers a $\Z^n$-action on $E^n$ 
by translations. 

Second, the notion of a support of a sheaf is lost in noncommutative
geometry, so noncommutative deformation destroys whatever algebraic
integrability that the action \eqref{tens} may have. It is sometimes
replaced by local analytic integrals (given e.g.\ by monodromy of certain
linear difference equations) but even then the orbits of the dynamics are
typically dense, see also Section \ref{s_inv} below.

\subsection{}

In noncommutative projective geometry, the 3-generator Sklyanin 
algebra, or the elliptic quantum $\bP^2$, occupies a special place. 
In this paper, we focus on this key special case and discuss 
the corresponding dynamics from several points of view, including 
an explicit linear algebra description of it, see Section
\ref{s_concrete}. This explicit description may be reformulated 
as addition on a 
moving Jacobian, generalizing the dynamics of \cite[\S 7]{KMNOY}.

In the first nontrivial case, we find the elliptic difference Painlev\'e
equation of \cite{Sakai}, the one that gives all other Painlev\'e equations
by degenerations and continuous limits.  A particularly detailed discussion
of this example may be found in Section \ref{s_ell_p}.  In particular, we
will see that in this case, our system of isomorphisms between moduli
spaces agrees (for sufficiently general parameters) with the corresponding
system of isomorphisms between rational surfaces considered by Sakai.

In the semiclassical limit, the elliptic quantum $\bP^2$ degenerates to a
Poisson structure on a commutative $\bP^2$, which
\cite{Tyurin,Bottacin,Hurtubise/Markman} induces a Poisson structure on
suitable moduli spaces of sheaves, and the moduli spaces we consider in the
commutative case are particularly simple instances of symplectic leaves in
these Poisson spaces.  In Section \ref{s_poisson}, we show that these
Poisson structures on moduli spaces carry over to the noncommutative setting.

\subsection{}

The principal 
results contained in this paper were obtained in September 2008 
during our stay at the CRM in Montreal. It is a special pleasure to 
thank John Harnad and Jacques Hurtubise for making this possible 
and to acknowledge their fundamental contribution to the subject
which is being deformed here in the noncommutative direction. 

We received valuable feedback, in particular, from D.~Kaledin 
and D.~Kazhdan during the first author's 2009 Zabrodsky at Hebrew
University, as well as from many other people on other occasions. 

AO thanks NSF for financial support under FRG 1159416. EMR was supported by
NSF grants DMS-0833464 and DMS-1001645.

\section{Blowups and Hecke modifications} 

\subsection{}

\subsubsection{}

In this paper, we work with one-dimensional sheaves on noncommutative
projective planes. They closely resemble their commutative ancestors, 
which we briefly review now. 

A coherent sheaf $\cM$ on $\bP^2$ is an object in the category 
\eqref{CohS} for $\bA=\C[x_0,x_1,x_2]$.  
A basic invariant of $\cM$ is its Hilbert 
polynomial 
$$
h_\cM(n) = \dim \bM_n\,, \quad n \gg 0\,.
$$
The dimension of $\cM$ is the degree of this polynomial, so for 
one-dimensional sheaves we have 
$$
h_\cM(n) = d n  + \chi
$$
where $d$ is the degree of the scheme-theoretic support of $\cM$ 
and $\chi$ is the 
Euler characteristic of $\cM$. The ratio $\chi/d$ is called the slope of 
$\cM$. Sheaves with 
$$
\slope(\cM') < \slope(\cM)
$$
for all proper subsheaves $\cM'$ are called \emph{stable}; the moduli 
spaces of stable sheaves are particularly nice. 

\subsubsection{}

We will be content with \emph{birational} group actions, hence it will be 
enough for us to 
consider open dense subsets of the moduli spaces formed by 
sheaves of the form 
$$
\cM = \iota_* L 
$$
where $\iota: C \hookrightarrow S$ is an inclusion of a smooth curve of 
degree $d$ and $L$ is a line bundle on $C$. All such sheaves are stable with 
$$
\chi = \deg L + 1 - g\,. 
$$
Here $g=(d-1)(d-2)/2$ is the genus of $C$. Their moduli space is a 
fibration over the base 
$$
B = \bP^{d(d+3)/2} \setminus \{\textup{singular curves}\}
$$
of nonsingular curves $C$ with the fiber $\Jac_{\deg L}(C)$, the Jacobian 
of line bundles of degree $\deg L$. In particular, this moduli space 
has dimension 
$$
\dim X = d^2 + 1 \,. 
$$

\subsubsection{}

Curves $C$ meeting a point $p\in \bP^2$ form a hyperplane in $B$. 
Incidence to $p$ may be rephrased in terms of the blowup 
$$
\Bl: S \to \bP^2 
$$
with center $p$. Namely, $C$ meets $p$ if and only if 
$$
C = \Bl(\wt{C})
$$
where $\wt{C}\subset S$ is a curve of degree
$$
\big[\wt{C}\big] = d \cdot \big[\textup{line}\big] - 
\big[\cE\big] \quad \in H_2(S,\Z) \,.
$$
Here $\cE = \Bl^{-1}(p)$ is the exceptional divisor of the blowup. 

Line bundles $\wt{L}$ on $\wt{C}$ may be pushed forward to $\bP^2$ 
to give sheaves that surject to the structure sheaf $\cO_p$ of $p$.
If $\wt{\cM}$ is such a line bundle viewed as a sheaf on $S$ then 
the sheaves 
\begin{equation}
  \cM = \Bl_*\wt{\cM}\,, \quad \cM' = \Bl_*\wt{\cM}(-\cE)\,,
\label{BlM}
\end{equation}
fit into an exact sequence of the form 
\begin{equation}
0 \to \cM' \to \cM \to \cO_p \to 0 \label{hecke} \,.
\end{equation}
When two sheaves $\cM$ and $\cM'$ differ by \eqref{hecke} one says
that one is a \emph{Hecke modification} of another. 
Thus Hecke modifications at $p$ correspond to twists by the 
exceptional divisor on the blowup with center $p$. For 
noncommutative algebras, the language of Hecke modifications will be 
more convenient. 

\subsection{}\label{class_poiss}

A fundamental fact that goes back to Mukai and Tyurin is that 
a Poisson structure $\omega^{-1}$ on a surface induces a Poisson 
structure on moduli of sheaves, see for example Chapter 10 of
\cite{HL}. Let 
$(\omega^{-1})$ denote the divisor of the Poisson structure. 
For $\bP^2$ this is a curve $E$ of degree $3$. Fix 
\begin{equation}
p_1,\dots,p_{3d} \in E
\label{p13d}
\end{equation}
that lie on a curve $d$, that is, 
$$
\sum_{i=1}^{3d} [p_i] = \cO_{\bP^2}(d)\big|_{E}  \in \Jac_{3d}(E) \,.
$$
Let 
$$
B' = \bP^{(d-1)(d-2)/2} \setminus \{\textup{singular curves}\} \subset B
$$
parametrize curves $C$ meeting \eqref{p13d}, or equivalently, curves 
$\wt{C}$ 
in 
$$
S=\Bl_{p_1,\dots,p_{3d}}\bP^2 
$$
of degree $d \cdot [\textup{line}] - \sum [\cE_i]$. It is by now a 
classical fact, see \cite{BeauvilleK3}, that the fibration 
\begin{equation}
\xymatrix{
\Pic(C) \ar@{^{(}->}[r]\ar[d]& X' \ar[d] \\
[C] \ar@{^{(}->}[r] & B' \\
}
\label{algint} 
\end{equation}
is Lagrangian and that these are the symplectic leaves of the Poisson 
structure on $X$. Further, the group $\Pic(S)$ acts on \eqref{algint}
preserving the fibers and 
the symplectic form. Here $\Pic(C)$ is a countable 
union of algebraic varieties that parameterize line bundles on $C$ 
of arbitrary degree. 

The noncommutative deformation will perturb this discrete integrable 
system. In particular, the points $p_i$ will have to move on the 
cubic curve $E$. The following model example illustrates this
phenomenon. 

\subsection{}\label{nc_examples}

\subsubsection{}

The effect of noncommutativity may be already seen in the affine
situation. Let $R$ be a noncommutative deformation of $\C[x,y]$
and let us examine the effect of Hecke modifications \eqref{hecke}
on $R$-modules. 

The point modules for $R$ are zero-dimensional modules.
These are annihilated by the two-sided ideal generated by commutators
in $R$, so this ideal has to be nontrivial for point modules to exist. 
We consider
\begin{equation}
R = \C\langle x,y \rangle / (xy-yx=\hbar y) \,.
\label{defR}
\end{equation}
Here is $\hbar$ is a parameter that measures the strength of the 
noncommutativity. Setting $\hbar=0$ one recovers the commutative ring 
$\C[x,y]$ with the Poisson bracket 
\begin{equation}
\{x,y\} = \lim_{\hbar \to 0} \frac{xy-yx}{\hbar} = y  \,.
\label{Poiss_y}
\end{equation}
The line $\{y=0\}$ is formed by $0$-dimensional leaves of this Poisson 
bracket, they correspond to point modules 
$$
\cO_{s} = R/\fm_s, \quad \fm_s=(y,x-s)\,, \quad s\in \C \,, 
$$
for $R$. All of them are annihilated by $y$ which generates the 
commutator ideal of $R$.

Let $M$ be an $R$-module of the form 
$$
M = R/f\,, \quad f=f_0(x)+f_1(x,y) \cdot y \in R\,, \quad f_0 \ne 0 \,, 
$$
The maps $M \to \cO_s$ factor through the map 
$$
M \to M/yM = \C[x]/f_0(x) 
$$
and hence correspond to the roots of $f_0(x)$. So far, this is entirely 
parallel to the commutative case, except there are a lot fewer points ---
those are confined to the divisor of the Poisson bracket
\eqref{Poiss_y}. 

\subsubsection{}

The following simple lemma shows Hecke correspondences move the points 
of intersection with this divisor.

\begin{Lemma}
Let $s_1,s_2,\dots$ be the roots of $f_0(x)$ and let $M'=\fm_s M$ be 
the kernel in the exact sequence 
$$
0 \to M' \to M \to \cO_{s_1} \to 0 \,. 
$$
Then $M'= R/f'$ where the roots of $f'_0(x)$ are $s_1-\hbar,s_2,s_3,\dots$. 
\end{Lemma}

\noindent
In particular, the iteration of Hecke correspondences give a chain
of submodules of the form 
$$
M \supset \fm_s M \supset \fm_{s-\hbar} \fm_s M \supset 
\fm_{s-2 \hbar} \fm_{s-\hbar} \fm_s M \supset \dots 
$$ 
A more general statement will be shown in Proposition \ref{p1}.

\subsubsection{}

For a noncommutative analog of the correspondence between 
Hecke modifications \eqref{hecke} and twists on the blowup \eqref{BlM}
we need to retrace geometric constructions in module-theoretic 
terms. 

Let $S$ be the blowup of an affine surface $S_0 = \Spec R$ 
with center in an ideal 
$I\subset R$. Sheaves on $S$ correspond to the quotient 
category \eqref{CohS} for
\begin{equation}
\bA = \bigoplus_{n=0}^\infty \bA_n\,, \quad \bA_0 = R \,,
\label{bA}
\end{equation}
and $\bA_n = I^n \subset R$. This quotient category 
$$
\Coh S = \Tails \bA 
$$
is informally known as the category of tails, the morphisms in it are 
$$
\Hom_{\textup{tails}} (\bM,\bM') = \underrightarrow{\lim} 
\Hom_{\textup{graded $\bA$-modules}} (\bM_{\ge k},\bM')
$$
where $\bM_{\ge k} \subset \bM$ is the submodule of elements of 
degree $k$ and higher. The push-forward $\Bl_*$ of sheaves is the functor 
\begin{equation}
\bM \mapsto \Bl_*\bM = \Hom_{\textup{tails}} (\bA,\bM) \in \Mod R \,. 
\label{Bl*}
\end{equation}
In the opposite direction, we have the pullback 
$
\Bl^* M = \bA \otimes_{R} M 
$
of modules as well as their proper transform 
$$
\Bl^{-1} M = \bigoplus I^n M  \in \Coh S \,. 
$$

\subsubsection{}

Now for a noncommutative ring $R$ as in \eqref{defR}, we look for a 
graded module $\bM$ over a graded algebra $\bA$ such that 
\begin{equation}
\Bl_*(\bM(n)) = \fm_{s-(n-1)\hbar} \cdots \fm_{s-\hbar} \fm_s M
\label{twistnc}
\end{equation}
where $\bM(n)_k = \bM_{n+k}$ is the shift of the grading and 
the pushforward is defined as in \eqref{Bl*}. Here $r\in R$ acts 
on $\phi \in \Hom_{\textup{tails}} (\bA,\bM)$ by 
$$
\left[ r \cdot \phi \right] (a)  = \phi(ar) \,. 
$$
The algebra $\bA$, known as Van den Bergh's noncommutative blowup 
\cite{ArtinSome}, 
is constructed as follows 
$$
\bA = \bigoplus_{n \ge 0} \big(T \fm_s\big)^n 
$$
where $T$ is a new generator subject to 
$$
T^{-1} \, r \, T  = y \, r \, y^{-1} \,, \quad \forall r \in R \,, 
$$
which means that 
$$
x T = T (x-\hbar) \,, \quad y T = T y \,, 
$$
and hence 
$$
\big(T \fm_s\big)^n  = T^n \, \fm_{s-(n-1)\hbar} \cdots \fm_{s-\hbar} \fm_s \,.
$$
It is easy to see that the $\bA$-module 
$$
\bM = \Bl^{-1} M = \bigoplus  T^n \,  \fm_{s-(n-1)\hbar} \cdots \fm_{s-\hbar} \fm_s  M 
$$
satisfies \eqref{twistnc} provided $M$ has no $0$-dimensional submodules
supported on $s,s-\hbar,\dots$. 

\subsubsection{}

To relate Hecke modifications to tensor products, we note 
that
$$
\Bl^{-1}_{s-\hbar}
\left(\fm_s M \right) = 
\bL \otimes_{\bA_{s}}
\left(\Bl^{-1}_s M \right)  
$$
where 
$$
\bL = T^{-1} \, \bA_{s}(1)  \in \Bimod(\bA_{s-\hbar},\bA_{s}) \,. 
$$
Here we indicated the centers of the blowup by subscripts $s$ and 
$s-\hbar$, respectively. 

The functor $\bL \otimes_{\bA_{s}}$ is the noncommutative version 
of $\cO_{S}(-\cE) \otimes$ and we see that it moves the 
center of the blowup by minus (to match the minus in 
$\cO_{S}(-\cE)$) the noncommutativity parameter $\hbar$.

\newpage

\section{Sheaves on quantum planes}

\subsection{}

One of the most interesting noncommutative surfaces is associated 
to the 3-dimensional \emph{Sklyanin algebra} $\bA$, which is 
a graded algebra, generated over $\bA_0 = \C$ by 
three generators $x_1,x_2,x_3\in \bA_1$ subject to three 
quadratic relations. 

The relations in $\bA$ may be written in the 
superpotential form 
$$
\frac{\partial}{\partial x_i} \, \mathcal{W} = 0 
$$
where 
$$
\mathcal{W} = a \,x_1 x_2 x_3 + b \,x_3 x_2 x_1 + \frac{c}{3} \sum x_i^3  
$$
and the derivative is applied cyclically, that is, 
$$
\frac{\partial}{\partial x_1} x_{i_0} \cdots x_{i_{p-1}} = 
\sum_{k=0}^{p-1} \delta_{1,i_{k}} x_{i_{k+1}} \cdots x_{i_{p-1+k}} \,,
$$
where the subsubscripts are taken modulo $p$. The parameters 
$a,b,c$ will be assumed generic in what follows. 

\subsection{}

The structure of $\bA$ is much studied, see for example \cite{SvdB}.  In particular, it is a Noetherian 
domain and 
$$
\sum_n \dim \bA_n \, t^n = (1-t)^{-3} \,. 
$$
By definition, the category $\Tails(\bA)$ is the category of coherent 
sheaves on a quantum $\bP^2$. The Grothendieck group of this category
is the same as the $K$-theory of $\bP^2$, that is 
$$
K\left(\Tails(\bA)\right) = \Z^3\,,
$$
corresponding to the three coefficients in the Hilbert polynomial. 
In particular, for $1$-dimensional sheaves, we have 
\begin{equation}
\dim \bM_n  = n \, \deg \bM   + \chi(\bM)  \,, \quad n\gg 0 \,. 
\label{Hilb_f}
\end{equation}

\subsection{} 
Modules $\bM$ such that $\dim \bM_n = 1$, $n\gg 1$, are called 
\emph{point modules} and play a very special role. Choosing a nonzero $v_n
\in \bM_n$ we get 
a sequence of points 
$$
p_n = \left(p_{1,n} : p_{2,n} : p_{3,n} \right) \in \bP^2 = 
\bP\left(\bA_1\right)^*
$$
such that 
$$
x_i \, v_n = p_{i,n} \, v_{n+1} \,. 
$$
The relations in $\bA$ then imply that the locus 
$$
\left\{(p_n,p_{n+1}) \right\}\in \bP^2 \times \bP^2 
$$
is a graph of an automorphism $p_{n+1}=\tau(p_n)$ of a plane cubic curve $E\subset\bP^2$, see \cite{ATV1}. The assignment 
\begin{equation}
\bM \mapsto p_0 \in E 
\label{Mtop}
\end{equation}
identifies $E$ 
with the moduli space of point modules $\bM$ and $\tau$ with the automorphism \footnote{
  Note that if $\tau'$ is any other automorphism of $E$ such that
  $\tau^{\prime 3}=\tau^3$, then the Sklyanin algebra associated to
  the pair $(E,\tau')$ has an equivalent category of coherent sheaves.
  Indeed, one has a natural isomorphism
  \[
  \bigoplus_n \bA_{3n}\cong \bigoplus_n \bA'_{3n},
  \]
  though this does not extend to an isomorphism $\bA\cong \bA'$.
  (Here 3 is the degree of the anticanonical bundle on $\bP^2$.)  This
  is why all key formulas below depend only on
  $\tau^3$.
}
$$
\tau: \bM \mapsto \bM(1) \,,
$$
of the shift of grading  $\bM(1)_{n} = \bM_{n+1}$. The inverse to 
\eqref{Mtop} is given by 
$$
p \mapsto \bA \big/ \bA \, p^\perp \,, 
$$
where $p^\perp \subset \bA_1$ is the kernel of 
$p \in \bP\left(\bA_1\right)^*$.

\subsection{}  

The action of $\bA$ on point modules factors through
the surjection in 
\begin{equation}
0 \to (\bE) \to \bA \to \bB \to 0 \,, 
\label{hom_to_B}
\end{equation}
where $\bE\in\bA_3$ is a distinguished normal
(in fact, central) element and $\bB$ is the 
twisted homogeneous coordinate ring of $E$. By definition, 
$$
\bB = B(E,\cO(1),\tau) = \bigoplus_{n\ge 0} H^0\left(E,\cL_0  \otimes \cdots 
\otimes \cL_{n-1} \right) \,,
$$
where $
\cL_k = \left(\tau^{-k}\right)^* \cO(1) \,. 
$
The multiplication 
$$
\textup{mult}_\tau : \bB_{n} \otimes \bB_{m} \to \bB_{n+m}
$$
is the usual multiplication precomposed with $\tau^{-m} \otimes 1$. 
See for 
example \cite{SvdB} for a general discussion of such algebras. 

The map 
$$
\Coh(E) \owns \cF \mapsto 
\Gamma(F) =\bigoplus_{n} H^0\left(E,\cF \otimes \cL_0  \otimes \cdots 
\otimes \cL_{n-1} \right)
$$
induces an equivalence between the category of coherent 
sheaves on $E$ and finitely generated graded $\bB$-modules
up-to torsion, see Theorem 2.1.5 in \cite{SvdB}. Note in 
particular that
\begin{equation}
  \label{GammaB}
  \bB(k) \cong 
\begin{cases}
\Gamma\left(\cL_{-k} \otimes \dots \otimes \cL_{-1}
\right) &k\ge 0\\
\Gamma\left(\cL_0\otimes\cdots \cL_{-k-1}\right)
&k<0.
\end{cases}
\end{equation}

\subsection{}

It is shown in \cite{ATV2}, Theorem 7.3, that the 
algebra $\bA\left[\bE^{-1}\right]_0$ is simple. If $\bM$ is
a $0$-dimensional $\bA$-module then $\bM\left[\bE^{-1}\right]_0$ 
is a finite-dimensional $\bA\left[\bE^{-1}\right]_0$-module, 
hence zero. It follows that 
any $0$-dimensional $\bA$-module has a filtration with point 
quotients.

\subsection{}\label{s_present}

Moduli spaces of stable $\bM \in \Tails \bA$ may be constructed
using the standard tools of geometric invariant theory, as in e.g.\
\cite{Nevins/Stafford}, or using the existence of an exceptional collection 
$$
 \bA,\bA(1),\bA(2) \in \Tails \bA\,,
$$
as in \cite{NvdB}. In any event, at least for generic parameters of $\bA$, 
the moduli space $\fM(d,\chi)$ of one-dimensional sheaves of degree
$d$ and Euler characteristic $\chi$ is irreducible of dimension
$$
\dim \fM(d,\chi) = d^2 + 1 \,.
$$
It is enough to 
see this in the commutative case, where a generic $\bM$ 
has a presentation of the form 
\begin{equation}
0 \to \bA(-2)^{d-\chi} \xrightarrow{\,\,L\,\,}
 \bA(-1)^{d-2\chi} \oplus \bA^\chi \to \bM \to 0 \,,
\label{resol}
\end{equation}
assuming 
$$
0 \le \chi \le \frac12 d \,, 
$$
see in particular \cite{BeauDet}. When $\frac12 d < \chi \le d$, 
the $\bA(-1)$ term moves from the generators to syzygies. 
The values of $\chi$ outside $[0,d]$ are obtained by a shift 
of grading. 

It follows that \eqref{resol} also gives a presentation of 
a generic stable one-dimensional $\bM$ for Sklyanin algebras. 

\subsection{} 

The letter $L$ is chosen in \eqref{resol} to connect with the 
so-called $L$-operators in theory of integrable systems. In \eqref{resol}, 
$L$ is a just a matrix with linear and quadratic entries in the 
generators $x_1,x_2,x_3 \in \bA_1$. The space of possible $L$'s, therefore, 
is just a linear space that needs to be divided by the action of 
\begin{multline*}
  \Aut \textup{Source}(L) \times \Aut \textup{Target}(L) \cong \\
  GL(d-\chi) \times GL(d-2\chi) \times GL(\chi) \ltimes \C^{\chi(d-2\chi)} \,.
\end{multline*}
In particular, we have a birational map 
\begin{equation}
\xymatrix{
\Mat(d\times d)^3/GL(d)\times GL(d)
\quad\ar@{-->}[r]& \quad 
\fM(d,0) \,,
}
\end{equation}
which is literally unchanged from the commutative situation.

\section{Weyl group action on parabolic sheaves}

\subsection{}

Our goal in this section is to examine the action of 
Hecke correspondences \eqref{hecke} on $1$-dimensional sheaves $\bM$
on quantum planes. 
For this, the language of parabolic sheaves will be convenient.

In what follows we assume $\bM\in \Tails \bA$ is stable $1$-dimensional without
$\bE$-torsion. This means the sequence 
\begin{equation}
0 \to \bM(-3) \xrightarrow{\,\,\bE\,\,} \bM \to \bM/\bE \bM \to 0 
\label{resMEM}
\end{equation}
is exact and comparing the Hilbert polynomials we see 
$\bM/\bE \bM$ has a filtration with $3\deg \bM$ point
quotients. A choice of such filtration 
$$
\bM=\bM_0 \supset \bM_1 \supset \bM_2 \supset \dots \supset 
\bM_{3 \deg \bM} = \bE \bM \cong \bM(-3)
$$
is called a \emph{parabolic} structure on $\bM$. 

\subsection{} 

Moduli spaces
$\fP(d,\chi)$ of parabolic sheaves 
may be constructed as in the commutative
situation. The forgetful map 
\begin{equation*}
\xymatrix{
\fP(d,\chi) 
\quad\ar@{-->}[r]& \quad 
\fM(d,\chi) \,.
}
\end{equation*}
is generically finite of degree $(3\deg\bM)!$ corresponding to the 
generic module $\bM/\bE \bM$ being a direct sum of nonisomorphic 
point modules. 

Since nonempty GIT quotients corresponding to different 
choices of stability parameters are canonically birational, we need
not be specific about fixing the stability conditions for parabolic 
sheaves.

\subsection{}\label{s_def_part} 
Given a parabolic module $\bM$, we denote by 
$$
\partial \bM =  \left(\bM_0\big/\bM_1,\dots,\bM_{3d-1}\big/\bM_{3d} \right)
\in E^{3d}  
$$
the isomorphism class of its point factors. 

In the commutative case, the sum of $\partial\bM$ in $\Pic(E)$ equals
$\cO(\deg M)$. The analogous noncommutative statement reads

\begin{Proposition}\label{p1}
Let $\bM$ be $1$-dimensional and have no $\bE$-torsion. Then 
\begin{equation}
\sum_{p\in\partial\bM} p = \cO(\deg(\bM))+ 3(\chi(\bM)-\deg \bM) \, 
\tau 
\,\, \in \Pic_{3d}(E) \,. 
\label{sum_p}
\end{equation}
\end{Proposition}

Here we identify the automorphism 
$\tau$ with the element $\tau(p)-p \in \Pic_0(E)$. 
This does not depend on the choice of $p\in E$. 

\begin{proof} Let 
$$
\dots \to F_1 \to F_0 \to \bM \to 0
$$
be a graded free resolution of a module $\bM$. The cohomology groups 
of 
$$
\dots \to \bB \otimes F_1 \to \bB \otimes F_0 \to 0 
$$
are, by definition, the groups $\Tor^i(\bB,\bM)$. The class of 
the Euler characteristic
$$
\left[\Tor(\bB,\bM)\right] = \sum (-1)^i \left[\Tor^i(\bB,\bM)\right]
\in K(\bB) \cong K(E) 
$$
may be computed using only the $K$-theory class of $\bM$. In fact, 
$$
c_1 \left(\left[\Tor(\bB,\bM)\right]\right) = 
\cO(\deg \bM) + 3(\chi (\bM)-\deg \bM -  \rk \bM) \, \tau \,. 
$$
It is enough to check this for $\bM = \bA(k)$, which follows from 
\eqref{GammaB}.

Alternatively, the groups $\Tor^i(\bB,\bM)$ may be computed from 
a free resolution of $\bB$. From \eqref{resMEM}, we find 
$$
\Tor^0(\bB,\bM) = \bM/\bE \bM\,,
$$
while all  higher ones vanish. The proposition follows. 
\end{proof}

\subsection{}

Let 
$$
\bW = S(3d) \ltimes \Z^{3d}
$$
be the extended affine Weyl group of $GL(3d)$. Weyl group actions 
on moduli of parabolic objects is a classic of geometric representation 
theory. In our context, the lattice subgroup may 
be interpreted as
$$
\Z^{3d} \cong \Pic \Bl_{p_1,\dots,p_{3d}} \bP^2 \Big/ \Pic \bP^2 \,,
$$
while $S(3d)$ acts on it by monodromy as the centers of the blowup 
move around. 

The group $\bW$ is generated by reflections 
$
s_0,\dots,s_{3d-1}
$
in the hyperplanes 
$$
\{a_{3d}=a_1-1\},\{a_1=a_2\},\dots,\{a_{3d-1}=a_{3d} \} \,, 
$$
together with the transformation 
$$
g \cdot (a_1,\dots,a_{3d}) \mapsto (a_2,\dots,a_{3d},a_{1}-1)\,. 
$$
The involutions $s_i$ satisfy the Coxeter relations
$$
(s_i s_{i+1})^3 = 1 
$$
of the affine Weyl group of $GL(3d)$ while $g$ 
acts on them as the Dynkin diagram automorphism 
$$
g \, s_i \, g^{-1} = s_{i-1}. 
$$
Here and above the indices are taken modulo $3d$.

\subsection{}

On the open locus where
$$
\bM/\bE \bM = \bigoplus_{i=1}^{3d} \cO_{p_i} 
$$
and all $p_i$'s are distinct, the symmetric group $S(3d)$ acts
on parabolic structures by permuting the factors. 

This extends
to a birational action of $S(3d)$ on $\fP(d,\chi)$. The closure of the 
graph of $s_k$ may be described as the nondiagonal component 
of the correspondence 
$$
\left\{(\bM,\bM')\, \big| \, \bM_i = \bM'_i, i \ne k \right\} \subset \fP \times
\fP \,. 
$$

\subsection{}
We define 
$$
g \cdot \bM = \bM_1
$$
with the parabolic structure 
$$
\bM_1 \supset \dots \supset \bM_{3d} \supset \bE \bM_1  \,, 
$$
where, as before, we assume that $\bM$ has no $\bE$-torsion. 
This gives a birational 
map 
$$
g: \fP(d,\chi) 
\to 
\fP(d,\chi-1) \,.
$$

\subsection{}\label{WonE}

We make $\bW$ act on $E^{3d}$ by 
\begin{equation}
\Z^{3d} \owns (a_1,a_2,\dots) \mapsto (\tau^{3a_1},\tau^{3a_2},\dots) 
\in \Aut(E^{3d}) 
\label{act_LE}
\end{equation}
while $S(3d)$ permutes the factors. Then we have
\begin{equation}
  \label{Etau}
  \partial \, \bE \bM = \partial \, 
\bM(-3) = (-1,\dots,-1) \cdot \partial \bM \,. 
\end{equation}

\begin{Theorem}\label{t1}
The transformations $s_1,\dots,s_{3d-1}$ and $g$ generate
an action of $\bW$ by birational transformations of $\bigsqcup_\chi
\fP(d,\chi)$. 
The map 
$$
\partial: \fP(d,\chi) \to E^{3d} 
$$ 
is equivariant with respect to this 
action. 
\end{Theorem}

\begin{proof}
Clearly, $s_1,\dots,s_{3d-1}$ generate the symmetric group $S(3d)$, 
as do their conjugates under the action of $g$. Setting 
$$
s_0 = g \, s_1 \, g^{-1} 
$$
one sees that $g$ permutes $s_0,\dots,s_{3d-1}$ cyclically, verifying 
all relations in $\Lambda$. Equivariance of $\partial$ follows 
from \eqref{Etau}. 
\end{proof}

\subsection{}\label{s_inv}

Evidently, the $\bA[\bE^{-1}]$ module $\bM[\bE^{-1}]$ is not 
changed by the dynamics, that is to say, its isomorphism class
is an invariant of the dynamics. From a dynamical viewpoint, however, 
this is not very useful information, since no 
reasonable moduli space for $\bA[\bE^{-1}]$-modules exists, which is 
just another way of stating the fact that generic orbits of our 
dynamical system are dense in the analytic topology. 

Local analytic integrals of the dynamics may be constructed in this setting
if a representation of the noncommutative algebra by linear difference
operators is given. (This will be done in \cite{Rains2}.) The monodromy of
the difference equation corresponding to a module is the required
invariant. A important virtue of such local invariants is their convergence
to algebraic invariants as the noncommutative deformation is removed, which
is very useful, for example, for the study of averaging of perturbations.

\section{A concrete description of the action}\label{s_concrete}

\subsection{}

The goal of this section to make the action in 
Theorem \ref{t1} as explicit as possible. Consider
the exact sequence 
$$
1 \to \bW_0 \to \bW \xrightarrow{\,\, \chi\,\,} \Z \to 1 
$$
where $\chi$ is the sum of entries on $\Z^{3d}$ and zero 
on $S(3d)$. We have 
$$
\chi(w \cdot \bM) = \chi(\bM) + \chi(w)\,,
$$
so the subgroup $\bW_0$ acts on $\fP(d,\chi)$ for any $\chi\in \Z$. 
Since all of them are birational, we can focus on one, for example
$$
X = \fP(d,d+1)\,. 
$$

\subsection{}

We will see that there is a diagram of maps, with birational top row 
\begin{equation}
\xymatrix{
X \ar@{-->}[r]\ar[dr]_{\partial}&  S^g\bP^2 \times E^{3d-1} \ar[d] \\
 &  E^{3d-1}\\
}\,, 
\label{bir}  
\end{equation}
where $g=\binom{d-1}{2}$ is the genus of a smooth curve of degree $d$ and $S^g \bP^2$ 
parametrizes unordered collections $D\subset \bP^2$ of $g$ points. We view 
the product $E^{3d-1}$ as embedded in $E^{3d}$ via
\begin{equation}
E^{3d-1} = \left\{ \sum_{i=1}^{3d} p_i = \cO(d) + 3 \tau\right\} \subset E^{3d} 
\label{sumco}
\end{equation}
This subset is $\bW_0$ invariant.

\subsection{}
The 
action of $\bW_0$ has a particularly nice description in terms of \eqref{bir} 
and it agrees with the action on $E^d$ already defined in Section \ref{WonE}.

The symmetric group $S(3d)$ permutes the points $p_i\in E$ and does nothing
to $D\subset \bP^2$. It remains to define the action of the 
lattice generators 
$$
\alpha_{ij} = \delta_i - \delta_j \in \Z^{3d}\,,
$$
where $\delta_i=(0,\dots,0,1,0,\dots,0)$ form the 
standard basis of $\Z^n$. We claim 
$$
\alpha_{ij} (D,P) = (D',P')
$$
where $P' = \alpha_{ij} P$ as in Section \ref{WonE}, while the points 
$D'$ are found from the following construction. 

Let $C \subset \bP^2$ be the degree $d$ curve that meets $D$ and
$(-\delta_j)\cdot P$. Because of the condition \eqref{sumco} such a curve
exists and is generically unique.  The divisor $D'\subset C$ is found from
$$
D + p_i = D' + p'_j \in \Pic(C)\,. 
$$
Again, since $D'$ is of degree $g(C)$, generically, there 
is a unique effective divisor satisfying this equation. 

\begin{Theorem}\label{t2} 
This defines a birational action of $\bW_0$ which is birationally 
isomorphic  
to the action from Theorem \ref{t1}. 
\end{Theorem}

Remark. The case $d=3$ of the above dynamical system was considered in
\cite[\S 7]{KMNOY} as a description of the elliptic Painlev\'e equation in
terms of the arithmetic on a moving elliptic curve.  Since we only consider
this in terms of birational maps, this only shows that our dynamics is
birationally equivalent to elliptic Painlev\'e; we will see in Section
\ref{s_ell_p} that (for generic parameters) the description in terms of
sheaves agrees {\em holomorphically} with elliptic Painlev\'e.

We break up the proof into a sequence of Propositions. 

\subsection{}

A general point $\bM \in X$ is of the form 
$\bM = \Coker L$ where 
$$
L: \bA(-1)^{d-1}  \to  \bA(1) \oplus \bA^{d-2} \,,
$$
see Section \ref{s_present}. 
Consider the submatrix 
$$
\Lb: \bA(-1)^{d-1}  \to  \bA^{d-2} \,. 
$$
and let $\dv\bM\subset \bP^2$ be the subscheme cut out 
by the maximal minors of $\Lb$. Generically,   
$$
\dv\bM = \binom{d-1}{2}\textup{ distinct points} \,. 
$$

\begin{Proposition}\label{p_div} 
The map
$$
X \owns \bM \mapsto \left(\dv \bM,\partial \bM\right) \in S^g \bP^2 \times 
E^{3d-1} \,,
$$
where  $E^{3d-1}\subset E^{3d}$ as in \eqref{sumco}, is birational. 
\end{Proposition}

\begin{proof}
The statement about $\sum p$ follows from \eqref{sum_p}. 
Since the source and the target have
the same dimension, it is enough to show that 
the map has degree $1$. For this, we may assume that 
$\bA$ is commutative, in which case the claim is 
classically known, see e.g.\ \cite{BeauDet}. 
\end{proof}

\subsection{}\label{s_div} 

In fact, in the commutative case, $\bM$ is generically a line 
bundle $\cL$ on a smooth curve $C=\supp \bM$. From its resolution, 
we see that $\cL(-1)$ has a unique section. The divisor of 
this section on $C$ is $\dv \bM$.

\subsection{}

Similarly,  
a general point $\bM\subset \fP(d,d)$ 
is of the form $\bM = \Coker L$, where 
$$
L: \bA(-1)^{d}  \to  \bA^{d} \,. 
$$
is a $d\times d$ matrix 
$L$ of linear forms. The matrix $L$ is unique 
up to left and right multiplication by elements
of $GL(d)$. The description of this $GL(d)\times GL(d)$ 
quotient is classically known \cite{BeauDet} and given by 
$$
L \mapsto (C,\cL)
$$
where 
$$
C=\{\det L = 0 \} \subset \bP^2
$$ 
is a degree $d$ curve cut out by the usual, commutative 
determinant and $\cL$ is the cokernel of the commutative morphism, viewed
as a sheaf on $C$. Generically, $C$ is smooth, $\cL$ is a
line bundle, and 
$$
g(C) = \binom{d-1}{2}\,, \quad \deg \cL = g(C) - 1 + d\,.
$$

\begin{Proposition}
The point modules $\cO_p$ in $\partial\bM$ correspond 
to points $p\in C \cap E$. 
\end{Proposition}

\begin{proof}
Let $f_1,\dots,f_d$ be the images in $\bM$ of the generators of
$\bA^{d}$. We may assume that $f_2,\dots,f_{d}$ are in 
the kernel of $\bM\to \cO_p$. The coefficient of $f_1$ in any 
relation among the $f_i$'s must belong to $\bA \, p^\perp$, 
whence the claim. 
\end{proof}

\subsection{}

Now suppose $\bM \in \fP(d,d+1)$, $p \in \partial \bM$, and define
$$
\bM_p = \Ker \left(\bM \to \cO_p\right) = \fm_p \bM  \,. 
$$
Then $\bM_p\in \fP(d,d)$ and we denote by $(C_p,\cL_p)$ be the curve 
and the line bundle that correspond to $\bM_p$.

\begin{Proposition}\label{p_meets_dv}
The curve $C_p$ meets $\dv \bM$. 
\end{Proposition}

\begin{proof}
Let $f$ and $g_1,\dots,g_{d-2}$ be the images in $\bM$ of the 
generators of $\bA(1)$ and $\bA^{d-2}$, respectively. We may 
chose them so that all $g_i$ are mapped to $0$ in $\cO_p$. 
Let 
$$
0 \to \bA(-1) \xrightarrow{\,\, 
  \begin{bmatrix}
    u_1 \\ u_2 
  \end{bmatrix}
\,\,}
 \bA^{\oplus 2} 
\xrightarrow{\,\, 
  \begin{bmatrix}
    l_1 \, \, l_2 
  \end{bmatrix}
\,\,}
\bA(1) \to \cO_p \to 0 
$$
be a free resolution. 
All relations in $\bM$ must be of 
the form 
$$
(r_1 l_1 + r_2 l_2) f + \sum s_i \, g_i =0\,, \quad r_1,r_2,s_i \in \bA_1 \,. 
$$
There are $d-1$ linearly independent 
relations like this and the coefficients $s_i$ in them form 
the matrix $\Lb$. 

We observe that $l_1 f,l_2 f, g_1,\dots,g_{d-2}$ generate $\bM_p$ and, for 
this presentation, the matrix $L_p$ has the 
block form 
\begin{equation}
  L_p=
\begin{bmatrix}
u & r \\
0 &  \Lb
\end{bmatrix}
\label{formL}
\end{equation}
The proposition follows. 
\end{proof}

\subsection{}

Denote
$$
\partial \bM = (p,p_2,\dots,p_{3d}) \,.
$$
Since $\bM_p=\fm_p \bM$ still surjects to $\cO_{p_i}$, $i\ge 2$, we have
$$
\partial \bM_p = (p',p_2,\dots,p_{3d}) 
$$
for some $p'\in E$ and from \eqref{sum_p} we see that 
$$
p' = \tau^{-3}(p) \,.
$$
{} From the proof of Proposition \ref{p_meets_dv} we note that 
$$
p' = \{u_1 = u_2 =0\} \,.
$$

\subsection{}

The following Proposition concludes the proof of Theorem \ref{t2} 

\begin{Proposition}
We have $\cL_p = \dv \bM + \cO(1)-p'$ in $\Pic(C_p)$. 
\end{Proposition}

\begin{proof}
This is a purely commutative statement, in fact, just 
a restatement of the remark in Section  \ref{s_div}. 
\end{proof}

\section{The elliptic Painlev\'e equation}\label{s_ell_p}

Let us consider the case $d=3$ in more detail.  In this case, we can be
fairly explicit about the moduli space, at least for sufficiently general
parameters.  We suppose that the $3d=9$ points of $\partial \bM$ are distinct
(and ordered), so that specifying a point in $\ParHilb(3,\chi)$ is
equivalent to specifying the corresponding point in $\Hilb(3,\chi)$.
We consider the case $\chi=1$, as in that case we need consider only one
shape of presentation.  (Of course $\chi=-1$ would equally as well, by
duality; the case $\chi=0$ is somewhat trickier, though the calculation
below of the action of the Hecke modifications implies a similar
description for that moduli space.)

We naturally restrict
our attention to semistable sheaves, and note that a sheaf in $\Hilb(3,1)$
is semistable iff it is stable, iff it has no proper subsheaf with positive
Euler characteristic.

\begin{Lemma}
Suppose the sheaf $\bM\in \Hilb(3,1)$ has a free resolution of the form
\[
0\to \bA(-2)^2\to \bA(-1)\oplus \bA\to \bM\to 0,
\]
or in other words is generated by elements $f$, $g$ of degree $1$ and $0$
satisfying relations
\[
v_1f+w_1g=v_2f+w_2g=0,
\]
with $v_1,v_2\in A_1$, $w_1,w_2\in A_2$.  If $\bM$ is stable, then $v_1$,
$v_2$ are linearly independent and there is no element $x\in A_1$ such that
$v_1x=w_1$ and $v_2x=w_2$.
\end{Lemma}

\begin{proof}
If $v_1$, $v_2$ are not linearly independent, then without loss of
generality we may assume $v_2=0$.  But then $w_2\ne 0$ (by injectivity) and
thus the submodule generated by the image of $\bA$ is the cokernel of the map
$w_2:\bA(-2)\to \bA$, and has Euler characteristic 1, violating stability.

Similarly, if the second condition is violated, then we may replace $f$ by
$f-xg$ and thus eliminate the dependence of the relations on $g$.  But then
$\bA$ is a direct summand of $\bM$, violating the condition that $\bM$ have rank
$0$.
\end{proof}

Remark.  One can show that every stable sheaf in $\Hilb(3,1)$ must have a
presentation of the above form, but for present purposes, we simply
restrict our attention to the corresponding open subset of the stable
moduli space, which by the following proof is projective.

\begin{Theorem}
  Suppose $p_1,\dots,p_9$ is a sequence of $9$ distinct points of $E$ such
  that $p_1+\cdots+p_9={\cal O}(3)-6\tau$.  Then the moduli space of stable
  sheaves in $\Hilb(3,1)$ with
  $\bM|_E\supset (p_1,\dots,p_9)$ is canonically isomorphic to the blowup
  of $\bP^2$ in the images of $p_1,\dots,p_9$ under the embedding $E\to
  \bP^2$ coming from ${\cal L}_1\sim {\cal O}(1)-3\tau$.
\end{Theorem}

\begin{proof}
  We may view the coefficients $v_1$, $v_2$, $w_1$, $w_2$ as global
  sections of line bundles on $E$; to be precise,
\[
v_1,v_2\in \Hom(\bA(-2),\bA(-1)) 
     \cong \Hom(\bB(-2),\bB(-1))
     \cong H^0({\cal L}_1)
\]
and
\[
w_1,w_2\in \Hom(\bA(-2),\bA)
       \cong \Hom(\bB(-2),\bB)
       \cong H^0({\cal L}_0\otimes {\cal L}_1).
\]
Now, $H^0({\cal L}_1)$ is $3$-dimensional, and $v_1$,$v_2$ are linearly
independent by stability, and we thus obtain a morphism from the moduli
space of stable sheaves with the above presentation to $\bP^2$.  Note also
that in this identification, the constraint on $\partial \bM$ reduces to a
requirement that
\[
w_1v_2-w_2v_1
\]
vanish at $p_1$,\dots,$p_9$.

We need to show that this morphism has $0$-dimensional fibers except
over the points $p_1$,\dots,$p_9$, where the fiber is $\bP^1$; this together
with smoothness will imply the identification with the blowup.

Note that a point $p\in E$ corresponds to the subspace
\[
H^0({\cal L}_1(-p))\subset H^0({\cal L}_1)
\]
and thus the cases to consider are those in which $v_1,v_2$ have no common
zero, those in which they have a single common zero not of the form $p_i$,
and those in which they have a common zero at $p_i$.  The key fact is the
following statement about global sections of line bundles on elliptic curves.

\begin{Lemma}
Let $v_1$, $v_2$ be linearly independent global sections of ${\cal L}_1$.
Then the map
\[
H^0({\cal L}_0\otimes {\cal L}_1)^2\to H^0({\cal L}_0\otimes {\cal L}_1^2)
\]
given by $(w_1,w_2)\mapsto v_2w_1-v_1w_2$ is surjective if $v_1$, $v_2$
have no common zero, and otherwise has image of codimension 1, consisting
of those global sections vanishing at said common zero.
\end{Lemma}

\begin{proof}
Consider the complex
\[
0\to H^0({\cal L}_0)\to H^0({\cal L}_0\otimes {\cal L}_1)^2\to H^0({\cal
  L}_0\otimes {\cal L}_1^2)\to 0,
\]
with the left map being $x\mapsto (v_1x,v_2x)$.  This has Euler
characteristic $3-6-6+9=0$, and is exact on the left, so it will suffice to
understand the middle cohomology.  Now, the kernel of the above determinant
map consists of pairs $(w_1,w_2)$ with $v_2w_1=v_1w_2$.  Assuming neither
of $w_1,w_2$ is 0 (which would clearly imply $w_1=w_2=0$), we find that
\[
\dv(v_2)+\dv(w_1)=\dv(v_1)+\dv(w_2).
\]

If $v_1$, $v_2$ have no common zero, we conclude that
\[
\dv(w_1)-\dv(v_1)=\dv(w_2)-\dv(v_2)
\]
is an effective divisor, and thus
\[
w_1/v_1 = w_2/v_2\in H^0({\cal L}_0).
\]
But this gives exactness in the middle.

Similarly, if $v_1$, $v_2$ both vanish at $p$, then the same reasoning
shows that
\[
w_1/v_1 = w_2/v_2\in H^0({\cal L}_0(p)).
\]
In particular, we find that the middle cohomology has dimension at most 1;
since the right map is clearly not surjective in this case, its image must
therefore have codimension 1 as required.
\end{proof}

In particular, if $v_1$, $v_2$ have no common zero, the map from pairs
$(w_1,w_2)$ to the corresponding determinant is surjective, and thus there
is a unique pair $(w_1,w_2)$ up to equivalence compatible with the
constraints on $\partial \bM$.  If $v_1$, $v_2$ have a common zero not of the
form $p_i$, then the map fails to be surjective, and the only allowed
determinant is 0.  We thus obtain a $1$-dimensional space of possible pairs
(modulo multiples of $(v_1,v_2)$), giving rise to a single equivalence
class of stable sheaves.  Finally, if $v_1$, $v_2$ have a common zero at
$p_i$, then the 1-dimensional space of allowed determinants pulls back to a
2-dimensional space of pairs $(w_1,w_2)$ modulo multiples of $(v_1,v_2)$,
and thus gives rise to a $\bP^1$ worth of stable sheaves.

It remains to show smoothness.  The tangent space to a sheaf with
presentation $v_1f+w_1g=v_2f+w_2g=0$ consists of the set of quadruples
$(v'_1,v'_2,w'_1,w'_2)$ such that
\[
v'_1w_2-v'_2w_1+v_1w'_2-v_2w'_1
\]
vanishes at $p_1$,\dots,$p_9$.  (More precisely, it is the quotient of this
space by the space of trivial deformations, induced by infinitesimal
automorphisms of $\bA(-2)^2$ and $\bA(-1)\oplus \bA$; stability implies
that the trivial subspace has dimension $4+5-1=8$, independent of $\bM$.)
It will thus suffice to show that this surjects onto $H^0({\cal L}_0\otimes
{\cal L}_1^2)$, since then the dimension will be independent of $\bM$ (and
equal to $18-8-(9-1)=2$).  If $v_1$, $v_2$ have no common zero, this
follows directly from the lemma.  If they have a common zero at $p$, but
$v_1w_2-v_2w_1\ne 0$, then since $p$ is equal to at most one $p_i$, we
conclude that one of $w_1$ or $w_2$ must not vanish at $p$, so again the
lemma gives surjectivity.

Finally, if $v_1$,$v_2$,$w_1$,$w_2$ all vanish at $p$, then
$v_1w_2-v_2w_1=0$.  But then we again find as in the lemma that
\[
w_1/v_1=w_2/v_2\in H^0({\cal L}_0),
\]
and this violates stability.
\end{proof}

Remark.  Presumably this result could be extended to remove the constraint
that the base points are distinct, except that the blowup will no longer be
smooth, since the base locus of the blowup is then singular.

We also wish to understand how the action of the affine Weyl group
$\Lambda_0\cong \tilde{A}_8$ translates to this explicit description of the
moduli space.  The action of $S_9$ is essentially trivial, as this simply
permutes the points $p_1$,\dots,$p_9$.  It remains to consider the
generator $s_0$.  This can be performed in two steps: first shift down
$p_1$ to obtain a sheaf with Euler characteristic 0 and
$\partial \bM = (p_2,\dots,p_9,\tau^{-3}(p_1))$, then shift up $p_9$ to
obtain a sheaf with Euler characteristic $1$ and
$\partial \bM = (\tau^3(p_9),p_2,\dots,p_8,\tau^{-3}(p_1))$ as required.

Let ${\cal O}_{p_1}$ have the free resolution
$$
0 \to \bA \xrightarrow{\,\, 
  \begin{bmatrix}
    u_1 \\ u_2 
  \end{bmatrix}
\,\,}
 \bA(-1)^{\oplus 2} 
\xrightarrow{\,\, 
  \begin{bmatrix}
    l_1 \, \, l_2 
  \end{bmatrix}
\,\,}
\bA \to \cO_{p_1} \to 0.
$$
Then there are two cases to consider.  If $\langle v_1,v_2\rangle \ne
\langle l_1,l_2\rangle$, then we may choose our generators $f$,$g$ of $\bM$
in such a way that $g$ maps to 0 in ${\cal O}_{p_1}$.  Then we have
relations of the form
\[
(r_{11}l_1+r_{21}l_2)f + v_1 g= (r_{21}l_1+r_{22}l_2)f + v_2 g=0,
\]
and the submodule generated by $l_1f$, $l_2f$, $g$ has a presentation
of the form
\[
0\to \bA(-2)^3\xrightarrow{L} \bA(-1)^3\to \bM_{p_1}
\]
with
\[
L=
\begin{bmatrix}
u_1 &r_{11}&r_{12}\\
u_2 &r_{21}&r_{22}\\
 0  & v_1  & v_2
\end{bmatrix}.
\]
Generically, $\det(L)$ has divisor $\tau^{-3}(p_1)+p_2+\cdots+p_9$, and
thus has rank $2$ at $p_9$.  Then suitable row and column operations
recover a new matrix of the above form except with $u'_1$, $u'_2$ vanishing
at $p_9$; thus $\bM_{p_1}\cong \bM'_{\tau^3(p_9)}$ for suitable $\bM'$ as
required.

This can fail in only two ways: either the rank of $L(p_9)$ could be
smaller than $2$, or the corresponding column could have rank only $1$.
Suppose first that $\det(L)\ne 0$.  As long as $\tau^{-3}(p_1)\ne p_9$,
we find that $p_9$ is a simple zero of $\det(L)$, so the rank cannot drop
below $2$.  If after the row and column operations we find that $u'_1$,
$u'_2$ are linearly dependent, then we find that $\det(L)$ must factor as a
product $uv$ with $u\in H^0({\cal L}_1(-p_9))$, $v\in H^0({\cal L}_1^2)$.
In particular, this can only happen if we have
\[
{\cal L}_1(p_i+p_j+p_9)\cong {\cal O}_E
\]
for some $2\le i<j\le 8$, or
\[
{\cal L}_1(\tau^{-3}p_1+p_i+p_9)\cong {\cal O}_E
\]
for some $2\le i\le 8$.

If $\det(L)=0$ on $E$, we may consider the cubic polynomial obtained by
viewing $L$ as a matrix over $\bP^2$, and observe that this must be the
equation of $E$.  In particular, since $E$ is smooth, we still cannot have
rank $<2$ at any point, and since $E$ is irreducible, $L$ cannot be made
block upper triangular.

We thus conclude that as long as the points
\[
\tau^{-3}(p_1),p_1,p_2,\dots,p_9,\tau^3(p_9)
\]
are all distinct, and no three add to a divisor representing ${\cal L}_1$,
then $s_0$ induces a morphism between the complement of the fiber over
$p_1$ in the original moduli space and the complement of the fiber over
$\tau^3(p_9)$ in the new moduli space.

It remains to consider the fiber over $p_1$.  In this case, we note that
$f$ maps to $0$ in ${\cal O}_{p_1}$, and thus we can no longer proceed as
above.  In a suitable basis, the relations of $\bM$ now read
\[
r_1 f + l_1 g = r_2 f + l_2 g = 0,
\]
and thus $\bM_{p_1}$ is generated by $f$, with the single relation
\[
(u_1 r_1 + u_2 r_2)f=0.
\]
In particular, we obtain a sheaf with presentation of the form
\[
0\to \bA(-3)\to \bA\to \bM_{p_1}\to 0.
\]
It remains only to show that every sheaf with such a presentation arises in
this way, and that we can recover the original sheaf from the presentation.

\begin{Lemma}
Let $p\in E$ be any point, cut out by linear equations $l_1=l_2=0$.  Then
the central element $\bE \in \bA_3$ can be expressed in the form
\[
\bE = l_1f_1+l_2f_2
\]
with $f_1,f_2\in \bA_2$, and this expression is unique up to adding a pair
\[
(v_1g,v_2g)
\]
to $(f_1,f_2)$, where $v_1,v_2,g\in \bA_1$ are such that $l_1v_2+l_2v_2=0$.
\end{Lemma}

\begin{proof}
Consider the map
\[
\bA_2^2\to \bA_3
\]
given by $(f_1,f_2)\mapsto l_1f_1+l_2f_2$.  Since $(l_1,l_2)$ cuts out a
point module, the image of this map must be codimension $1$, and since $\bE$
annihilates every point module, the image contains $\bE$.  Uniqueness follows
by dimension counting, since the specified kernel is $3$-dimensional.
\end{proof}

We conclude that for any element of $\Lambda_0$, there is a finite
collection of linear inequalities on the base points which guarantee that
both the domain and range are blowups of $\bP^2$, and the element of
$\Lambda_0$ acts as a morphism.  In particular, we see that it suffices to
have
\[
\tau^{3k}p_i\ne p_j
\]
for $i<j$, $k\in \Z$, and
\[
{\cal L}_{3l+1}\not\cong {\cal O}_E(p_i+p_j+p_k)
\]
for $i<j<k$, $l\in \Z$, in order for the entire group $\Lambda_0$ to act as
isomorphisms between blowups of $\bP^2$.

In particular, we find that the translation subgroup of $\Lambda_0$ acts in
the same way as the translation subgroup of $\tilde{E}_8$ in Sakai's
description of the elliptic Painlev\'e equation.  Note that both groups
have the same rank, and by comparing determinants under the intersection
form in the commutative limit, we conclude that the translation subgroup of
$\Lambda_0$ has index $3$ in the translation subgroup of $\tilde{E}_8$.  It
is straightforward to see that we can generate the entire lattice by
including the operation
\[
\bM\mapsto g^3\bM(1),
\]
though it is more difficult to see how this acts in terms of presentations
of sheaves.

We note in passing that the above calculation of $\bM_{p_i}$ shows that the
$-1$-curve corresponding to the fiber over $p_i\in \bP^2$ can be described
as the subscheme of moduli space where the sheaf $\bM_{p_i}$ of Euler
characteristic 0 has a global section; this should be compared to the
cohomological description of $\tau$-divisors in \cite{Arinkin/Borodin2}.
In fact, every $-1$-curve on the moduli space has a similar description:
act by a suitable element of $\Lambda_{E_8}$, then ask for the Hecke
modification at $p_1$ to have a global section.

We also note that the results of \cite[Thm. 7.1]{Rains:hitchin} suggest
that one should consider the moduli space of stable sheaves $\bM$ on $\bA$
such that $h(\bM)=3rt+r$, and
\[
\bM|_E=({\cal O}_{p_1}\oplus \cdots {\cal O}_{p_9})^r,
\]
where now $p_1+\cdots+p_9-{\cal O}(3)+6\tau$ is a torsion point of order
$r$.  This moduli space remains $2$-dimensional, and is expected to again
be a $9$-point blowup of $\bP^2$.  (A variant of this will be considered in
\cite{Rains2}.)

\section{Poisson structures}\label{s_poisson}

\subsection{}
The Poisson structure on the moduli space of sheaves on a commutative
Poisson surface has a purely categorical definition (originally constructed
by \cite{Tyurin}, shown to satisfy the Jacobi identity for vector bundles
in \cite{Bottacin}, and extended to general sheaves of homological
dimension 1 in \cite{Hurtubise/Markman}).  This definition can be carried
over to the noncommutative case, and we will see that the analogous
bivector again gives a Poisson structure, and the Hecke modifications again
act as symplectomorphisms.  The main qualitative difference in the
noncommutative case is (as we have seen) that the Hecke modifications are
no longer automorphisms of a given symplectic leaf, but rather give maps
between related symplectic leaves.

Tyurin's construction in the commutative setting relies on the
observation that the tangent space at a sheaf $M$ on a Poisson surface $X$,
or equivalently the space of infinitesimal deformations, is given by the
self-$\Ext$ group
\[
\Ext^1(M,M).
\]
By Serre duality, the cotangent space is given by
\[
\Ext^1(M,M\otimes\omega_X).
\]
This globalizes to general sheaves such that $\dim\End(M)=1$; the cotangent
sheaf on the moduli space is given by $\Ext^1(M,M\otimes\omega_X)$, where
$M$ is the universal sheaf on the moduli space.  
% (Note that this only
% exists \'etale-locally, and the moduli space is only an algebraic space,
% but this is not particularly important.) 
A nontrivial Poisson structure on
$X$ corresponds to a nonzero morphism $\wedge^2\Omega_X\to {\cal O}_X$, or
equivalently to a nonzero morphism $\alpha:\omega_X\to {\cal O}_X$.
We thus obtain a map
\[
\Ext^1(M,M\otimes\omega_X)\xrightarrow{1\otimes \alpha} \Ext^1(M,M),
\]
and a bilinear form on $\Ext^1(M,M\otimes\omega_X)$.  One can then
show \cite{Bottacin,Hurtubise/Markman} that this bilinear form
induces a Poisson structure.  In addition, the resulting Poisson variety
has a natural foliation by {\em algebraic} symplectic leaves: if $C_\alpha$
is the curve $\alpha=0$, then for any sheaf $M_\alpha$ on $C_\alpha$, the
(Poisson) subspace of sheaves $M$ with $M\otimes {\cal O}_{C_\alpha}\cong
M_\alpha$, $\Tor_1(M,{\cal O}_{C_\alpha})=0$ is a smooth symplectic leaf.
%(More generally, the symplectic leaves are the subspaces with fixed
%$M\otimes^{\bf L}{\cal O}_{C_\alpha}$, see \cite{Rains:poisson}.)

In our noncommutative setting, there is again an analogue of Serre duality;
one finds that $H^2(\bA(-3))\cong \C$ (just as in the commutative
case), and for any $\bM$, $\bM'$ we have canonical pairings
\[
\Ext^i(\bM',\bM(-3))\otimes \Ext^{2-i}(\bM,\bM')\to H^2(\bA(-3)),
\]
Moreover, these pairings are (super)symmetric in the following sense.  If
$\alpha\in \Ext^i(\bM',\bM(-3))$, $\beta\in \Ext^{2-i}(\bM,\bM')$, then
\[
\langle \alpha,\beta\rangle
=
(-1)^i\langle \beta(-3),\alpha\rangle
=
(-1)^i\langle \beta,\alpha(3)\rangle.
\]
In addition, the pairing factors through the Yoneda product, in that
\[
\langle \alpha,\beta\rangle = \langle \alpha\cup\beta,1\rangle =:\tr(\alpha\cup\beta).
\]

As in the commutative case, infinitesimal deformations of
$\bM$ are classified by $\Ext^1(\bM,\bM)$, and the map
\[
\Ext^1(\bM,\bM(-3))\xrightarrow{E\cup\,\textup{--}} \Ext^1(\bM,\bM)
\]
induces a skew-symmetric pairing
\[
\Ext^1(\bM,\bM(-3))\otimes \Ext^1(\bM,\bM(-3))
\xrightarrow{\tr(\textup{--}\,\cup E\cup\,\textup{--} )}
H^2(\bA(-3))\cong \C,
\]
and this should be our desired Poisson structure.

Note here that the Poisson structure depends on the choice of $E$ and the
choice of automorphism $H^2(\bA(-3))\cong \C$; both are unique up to a
scalar, and only the product of the scalars matters.  The cohomology long
exact sequence associated to
\[
0\to \bA(-3)\xrightarrow{E} \bA\to \bB\to 0
\]
induces (since $H^1(\bA)=H^2(\bA)=0$) an isomorphism
\[
H^1(\bB)\cong H^2(\bA(-3)),
\]
depending linearly on $E$, so that the composition
\[
H^1(\bB)\cong H^2(\bA(-3))\cong \C
\]
scales in the same way as the Poisson structure.  In other words, the
scalar freedom in the Poisson structure corresponds to a choice of
isomorphism $H^1(\bB)\cong \C$; the canonical equivalence $\Tails \bB\cong
\Coh(E)$ turns this into an isomorphism $H^1({\cal O}_E)\cong \C$, or
equivalently a choice of nonzero holomorphic differential on $E$.

We will see that this construction remains Poisson in the noncommutative
setting, and that the description of the symplectic leaves carries over
mutatis mutandum.  Note that in this section, we will refer to the ``moduli
space of simple sheaves on $\bA$'', where a sheaf is {\em simple} if
$\End(\bM)=\C$ (in the commutative setting, this is a weakened form of the
constraint that a sheaf is stable).  One expects following
\cite{Altman/Kleiman} that this should be a quasi-separated algebraic space
$\fM_{\bA}$.  Per \cite{Rains:poisson}, a Poisson structure on such a space
is just a compatible system of Poisson structures on the domains of \'etale
morphisms to the space; in the moduli space setting, we must thus assign a
Poisson structure to every formally universal family of simple sheaves on
$\bA$.  The above bivector is clearly compatible, so will be Poisson iff it
is Poisson on every formally universal family.  Any statement below about
$\fM_{\bA}$ should be interpreted as a statement about formally universal
families in this way.

We will sketch two proofs of the following result below.

\begin{Theorem}
  The above construction defines a Poisson structure on $\fM_{\bA}$,
  and on the open subspace of sheaves transverse to $E$ (i.e., such that
  $\Tor_1(\bM,\bB)=0$), the fibers of the map $\bM\to \bM\otimes \bB\in \Coh(E)$
  are unions of (smooth) symplectic leaves of this Poisson structure.
\end{Theorem}

Remark.  One expects that, as in \cite{Rains:poisson}, one should have a
covering by algebraic symplectic leaves even without the transversality
assumption; in general, the symplectic leaves should be the preimages of
the {\em derived} restriction $\bM\to \bM\otimes^{\bL}\bB$, taking sheaves
on $\bA$ to the derived category of $\Coh(E)$.

We should note here that in the case $\bM$ is torsion-free (and stable), an
alternate construction of a Poisson structure was given in
\cite{Nevins/Stafford}; their Poisson structure is presumably a constant
multiple of the Tyurin-style Poisson structure.  

\subsection{}

Although the above construction is somewhat difficult to deal with
computationally (but see below), it has significant advantages in terms of
functoriality.  In particular, it is quite straightforward to show that
Hecke modifications give symplectomorphisms on the relevant symplectic
leaves.  Curiously, the argument ends up depending crucially on
noncommutativity!

With an eye to future applications, we consider a generalization of Hecke
modifications as follows.  Let $\bM$ be a simple $1$-dimensional sheaf on
$\bA$.  We define the ``downward pseudo-twist'' at $p\in E$ of $\bM$ to be
the kernel of the natural map $\bM\to \bM\otimes {\cal O}_p$; similarly, the
``upward pseudo-twist'' is the universal extension of ${\cal O}_p\otimes
\Ext^1({\cal O}_p,\bM)$ by $\bM$.  If the restriction $\bM|_E$ of $\bM$ to $E$
(i.e., $\bM\otimes \bB$, viewed as a sheaf on $E$) is not equal to the sum over
$p$ of $\bM\otimes {\cal O}_p$, then one could consider some other natural
modifications along these lines; in the commutative case, these correspond
to twists by line bundles on iterated blowups in which we have blown up the
same point on $E$ multiple times.  These will always be limits of the above
operations, so will again be symplectic by the limiting argument considered
below.

\begin{Proposition}
  The two pseudo-twists define (inverse) birational maps between symplectic
  leaves of the open subspace of $\fM_{\bA}$ classifying 1-dimensional
  sheaves transverse to $E$.  Where the maps are defined, they are
  symplectic.
\end{Proposition}

Remark. Note that we need merely prove that the morphisms preserve the
above bivector; this can be verified independently of whether the bivector
satisfies the Jacobi identity.  In addition, it suffices to prove that the
pseudo-twists are Poisson on suitable open subsets of the moduli space.

\begin{proof}
We first consider the downward pseudo-twist $\bM'$ of $\bM$, corresponding to
the short exact sequence
\[
0\to \bM'\to \bM\to \bM\otimes {\cal O}_p\to 0.
\]
We impose the additional conditions that 
\[
\Hom(\bM',{\cal O}_p)=\Hom(\bM,{\cal O}_p(-3))=0.
\]
Observe that this is really just a condition on the commutative sheaf
$\bM|_E$, stating that it is 0 near $\tau^{-3}(p)$, and near $p$ is a sum of
copies of ${\cal O}_p$.  Indeed, the first condition is precisely that
$\Hom(\bM,{\cal O}_p(-3))=0$ and implies $\Tor_1(\bM,\bB)=0$, while the second
condition follows from the four-term exact sequence
\[
0\to (\bM\otimes {\cal O}_p)(-3)\to \bM'|_E\to \bM|_E\to \bM\otimes {\cal O}_p\to
0.
\]

Note that if $\tau^{-3}(p)=p$, then the above conditions imply $\bM'\cong \bM$,
and thus eliminate any interesting examples of pseudo-twists.  Of
course, since $E$ is smooth, $\tau^{-3}(p)=p$ iff $\tau^3=1$, and this is
equivalent to the existence of an equivalence $\Tails \bA\cong \Coh(\bP^2)$.
Away from the commutative case, the conditions are not particularly hard to
satisfy; in particular, the generic sheaf in any component of the moduli
space of $1$-dimensional sheaves will satisfy this condition at every
point of $E$.

By Serre duality, we have $\Ext^2(\bM,{\cal O}_p(-3))\cong \Hom({\cal
  O}_p,\bM)=0$ and similarly $\Ext^2(\bM',{\cal O}_p)=0$ It then follows by an
Euler characteristic calculation that $\Ext^1(\bM,{\cal
  O}_p(-3))=\Ext^1(\bM',{\cal O}_p)=0$ as well.  Since $\bM\otimes {\cal O}_p$
is a sum of copies of ${\cal O}_p$, we find that the natural maps
\begin{align}
\Ext^i(\bM',\bM')&\to \Ext^i(\bM',\bM),\notag\\
\Ext^i(\bM,\bM'(-3))&\to \Ext^i(\bM,\bM(-3))\notag
\end{align}
are isomorphisms.  (In particular, $\bM'$ is simple iff $\bM$ is simple.) By
Serre duality, the same applies to
\begin{align}
\Ext^i(\bM,\bM)&\to \Ext^i(\bM',\bM),\notag\\
\Ext^i(\bM,\bM'(-3))&\to \Ext^i(\bM',\bM'(-3)).\notag
\end{align}
By the functoriality of $\Ext$, we find that the compositions
\[
\Ext^1(\bM,\bM'(-3))\cong \Ext^1(\bM',\bM'(-3))\xrightarrow{E} \Ext^1(\bM',\bM')\cong
\Ext^1(\bM',\bM)
\]
and
\[
\Ext^1(\bM,\bM'(-3))\cong \Ext^1(\bM,\bM(-3))\xrightarrow{E} \Ext^1(\bM,\bM)\cong
\Ext^1(\bM',\bM)
\]
agree, and thus the induced isomorphism
\[
\Ext^1(\bM,\bM)\cong \Ext^1(\bM',\bM')
\]
respects the Poisson structure.

It remains only to show that this isomorphism is the differential of the
pseudo-twist.  Note that the pseudo-twist is only a morphism on
the strata of the moduli space with fixed $\dim\Hom(\bM,{\cal O}_p)$.  Thus
we need only consider those classes in $\Ext^1(\bM,\bM)$ which preserve this
dimension.  In other words, we must consider extensions
\[
0\to \bM\to \bN\to \bM\to 0
\]
which remain exact when tensored with ${\cal O}_p$.  Then the corresponding
extension $\bN'$ of $\bM'$ by $\bM'$ is the kernel of the natural map $\bN\to \bN\otimes
{\cal O}_p$.  That both extensions have the same image in $\Ext^1(\bM',\bM)$
follows from exactness of the complex
\[
0\to \bM'\to \bM\oplus \bN'\to \bN\otimes \bM'\to \bM\to 0,
\]
(the two extensions are the cokernel of the map from $\bM'$ and the kernel of
the map to $\bM$), and this is the total complex of a double complex with
exact rows.

Note that we also have $\Ext^*({\cal O}_p,\bM)=0$, and thus the connecting
map
\[
\Hom({\cal O}_p,\bM\otimes {\cal O}_p))\to \Ext^1({\cal O}_p,\bM')
\]
is an isomorphism, implying that $\bM$ is the upward pseudo-twist of
$\bM'$.  Since we can restate the conditions on $\bM$, $\bM'$ in terms of
$\bM'|_C$, we find that the upward pseudo-twist is also Poisson.

In fact, the hypotheses on $\bM$, $\bM'$ are significantly stronger than
necessary.  The point is that once we constrain $\dim\Hom(\bM,{\cal O}_p)$,
the further constraints in the above argument are dense open conditions.
If we replace this by the weaker open condition that $\bM'$ is simple, we
still obtain a morphism between Poisson spaces.  The failure of such a
morphism to be Poisson is measured by a form on the cotangent sheaf, which
by the above argument vanishes on a dense open subset, and thus vanishes
identically.
\end{proof}

Remark. This limiting argument also lets us deduce the commutative case
from the noncommutative case, though in the commutative setting we can also
use an interpretation involving twists on blowups, see
\cite{Rains:poisson}; this actually works for arbitrary sheaves of
homological dimension 1, and presumably the same holds in the
noncommutative setting.  The above argument fails for torsion-free sheaves,
however, as does the fact that the upward and downward pseudo-twists are
inverse to each other.

\subsection{}

We now turn our attention to showing that the above actually defines a
Poisson structure, i.e., that the corresponding biderivation on the
structure sheaf satisfies the Jacobi identity.  Unfortunately, the existing
arguments in the commutative setting involve working with explicit \v{C}ech
cocycles for extensions of vector bundles; while both \v{C}ech cocycles and
vector bundles have noncommutative analogues, neither is particularly easy
to compute with.  It turns out, however, that in many cases, we can reduce
the computation of the pairing to a computation on the commutative curve
$E$.  (In fact, combined with the construction of \cite{Hurtubise/Markman},
this is enough to verify the Jacobi identity in general.)

We assume here that $\bM$ is a simple sheaf transverse to $E$; we also
assume $\bM/\bE \bM\ne 0$.  (In our case we could equivalently just assume
$\bM\ne 0$, but this is the form in which the condition appears below; for
commutative surfaces, the two conditions are not equivalent, and the
conditions are likely to deviate from each other for more general
noncommutative surfaces as well.)  The map giving the Poisson structure
then fits into a long exact sequence
\begin{align}
0&\to \Hom(\bM,\bM)\to \Hom(\bM,\bM/\bE \bM)\to \Ext^1(\bM,\bM(-3))\notag\\
 &\to \Ext^1(\bM,\bM)\to \Ext^1(\bM,\bM/\bE \bM)\to \Ext^2(\bM,\bM(-3))\to 0,\notag
\end{align}
where $\Hom(\bM,\bM(-3))\subset \Hom(\bM,\bM)$ is trivial since $\Hom(\bM,\bM)\cong \C$
injects in $\Hom(\bM,\bM/\bE \bM)$, and $\Ext^2(\bM,\bM)=0$ by duality.  Now,
\[
\Hom(\bM,\bM/\bE \bM)\cong \Hom_\bB(\bM/\bE \bM,\bM/\bE \bM),
\]
and may thus be computed entirely inside $\Coh(E)$, so via commutative
geometry.  Since the sequence is essentially self-dual, we should also
expect to have $\Ext^1(\bM,\bM/\bE \bM)\cong \Ext^1_\bB(\bM/\bE \bM,\bM/\bE \bM)$.  We can make this
explicit as follows: a class in $\Ext^1(\bM,\bM/\bE \bM)$ is represented by an
extension
\[
0\to \bM/\bE \bM\to \bN\to \bM\to 0,
\]
and since $\bM$ is $\bB$-flat, this induces an extension
\[
0\to \bM/\bE \bM\to \bN/\bE \bN\to \bM/\bE \bM\to 0,
\]
and pulling back recovers the original extension.  Conversely, any
extension of $\bM/\bE \bM$ by $\bM/\bE \bM$ over $\bB$ can be viewed as an extension of
$\bM/\bE \bM$ by $\bM/\bE \bM$ in $\Tails \bA$, and pulled back to an extension of $\bM$ by
$\bM/\bE \bM$ which restricts back to the original extension.  In other words,
``tensor with $\bB$'' and ``pull back'' give inverse maps as required.

Since the map $R\Hom(\bM,\bM(-3))\to R\Hom(\bM,\bM)$ in the derived category is
self-dual (subject to our choice of isomorphism $H^2(\bA(-3))\cong \C$), it
follows that the corresponding exact triangle is self-dual, and thus that
the remaining maps
\[
R\Hom(\bM,\bM)\to R\Hom(\bM,\bM/\bE \bM)\cong R\Hom_\bB(\bM/\bE \bM,\bM/\bE \bM)
\]
and
\[
R\Hom_\bB(\bM/\bE \bM,\bM/\bE \bM)\cong R\Hom(\bM,\bM/\bE \bM)\to R\Hom(\bM,\bM(-3))[1]
\]
in the exact triangle are dual.  In particular, it follows that we have a
commutative diagram
\[
\begin{CD}
\Ext^1(\bM,\bM/\bE \bM)@>\sim >> \Ext^1_\bB(\bM/\bE \bM,\bM/\bE \bM)@>\tr>> H^1({\cal O}_E)\\
@VVV @. @VV\sim V\\
\Ext^2(\bM,\bM(-3))@>\tr>> H^2(\bA(-3)) @>\sim >> \C
\end{CD}
%
%  \Ext^1(\bM,\bM/\bE \bM)\cong \Ext^1_\bB(\bM/\bE \bM,\bM/\bE \bM)\to^{\tr} H^1({\cal O}_E)
%      v                                        |~
%  \Ext^2(\bM,\bM(-3))\to^{\tr} H^2(\bA(-3))          \cong \C
%
\]
Since the map from $\Ext^1(\bM,\bM/\bE \bM)$ to $\Ext^2(\bM,\bM(-3))$ is surjective, to
compute the trace of any class in $\Ext^2(\bM,\bM(-3))$, we need simply choose
a preimage in $\Ext^1(\bM,\bM/\bE \bM)$, interpret it as an extension of sheaves on
the commutative curve $E$, and take the trace there.

Since we need only consider classes in $\Ext^2(\bM,\bM(-3))$ that arise via the
Yoneda product, it will be particularly convenient to use the Yoneda
interpretation of such classes via $2$-extensions.  If $\bN'$ is an extension
of $\bM$ by $\bM$ and $\bN$ is an extension of $\bM$ by $\bM(-3)$, then $\bN\cup \bN'$ is
represented by the $2$-extension
\[
0\to \bM(-3)\to \bN\to \bN'\to \bM\to 0,
\label{two-ext-a}
\]
where $\bN\to \bN'$ is the composition $\bN\to \bM\to \bN'$.  Recall that two
$2$-extensions are equivalent iff the complexes $\bN\to \bN'$ are quasi-isomorphic.
The functoriality of $\Ext^2(\textup{--}\,,\,\textup{--})$ is expressed via pullback and
pushforward, as appropriate; the connecting maps are more complicated, but
are again amenable to explicit description, see \cite{Mitchell}.

In our case, we have the following.  The pushforward of \eqref{two-ext-a}
under the map $\bM(-3)\xrightarrow{\bE } \bM$ has the form
\[
0\to \bM\to \bN''\to \bN'\to \bM\to 0,
\label{two-ext-b}
\]
where $\bN''\cong (\bN\oplus \bM)/\bM(-3)$.  Since $\Ext^2(\bM,\bM)=0$, this
$2$-extension is trivial, and thus there exists a sheaf $\bZ$ and a
filtration
\[
0\subset \bZ_1\subset \bZ_2\subset \bZ
\]
such that the sequence
\[
0\to \bZ_1\to \bZ_2\to \bZ/\bZ_1\to \bZ/\bZ_2\to 0
\]
agrees with \eqref{two-ext-b}, or equivalently such that
\[
0\to \bM\to \bN'\to \bM\to 0
\]
is the pushforward under $\bN''\to \bM$ of an extension
\[
0\to \bN''\to \bZ\to \bM\to 0.
\]
It follows that the $2$-extension \eqref{two-ext-a} is equivalent to
\[
0\to \bM(-3)\to \bZ'\to \bZ\to \bM\to 0,
\]
where $\bZ'$ is the pullback of $\bN$ under $\bN''\to \bM$.  Now, since $\bN''$ was
itself obtained by pushing $\bN$ forward, we have $\bN\subset \bN''$ in a natural
way, giving $\bN\subset \bZ',\bZ$ in compatible ways.  Quotienting by this gives
an equivalent $2$-extension
\[
0\to \bM(-3)\xrightarrow{\bE } \bM\to \bZ/\bN\to \bM\to 0,
\]
expressing \eqref{two-ext-a} as the image under the connecting map of
\[
0\to \bM/\bE \bM\to \bZ/\bN\to \bM\to 0.
\]
The corresponding class in $\Ext^1_\bB(\bM/\bE \bM,\bM/\bE \bM)$ is then obtained by tensoring
with $\bB$:
\[
0\to \bM/\bE \bM\to \bZ/(\bE \bZ+\bN)\to \bM/\bE \bM\to 0.
\]

It will be helpful to think of this last extension in a slightly different
way.  Since $\Tor_1(\bB,\bM)=0$, the quotient $\bZ/\bE \bZ$ inherits a filtration
\[
0\subset \bM/\bE \bM\subset \bN''/\bE \bN''\subset \bZ/\bE \bZ\to 0,
\]
so to compute $\bZ/(\bE \bZ+\bN)$, we need only understand the map
$\bN\to \bN''/\bE \bN''$.  Since $\bM(-3)\subset \bN$ maps to $\bE \bM\subset \bE \bN''$, the map
$\bN\to \bN''/\bE \bN''$ factors through the natural map $\bN\to \bM$, and then through
the quotient map $\bM\to \bM/\bE \bM$.  In other words, the map $\bN\to \bN''/\bE \bN''$
precisely gives a splitting of the short exact sequence
\[
0\to \bM/\bE \bM\to \bN''/\bE \bN''\to \bM/\bE \bM\to 0.
\]

Note in particular that the pairing of $\bN$ and $\bN'$ depends only on the two
extensions $\bN''$, $\bN'\in \Ext^1(\bM,\bM)$ and a splitting of $\bN''/\bE \bN''$.
We need simply combine $\bN''$, $\bN'$ into a filtered sheaf, quotient by $\bE $,
then mod out by the submodule $\bM/\bE \bM$ coming from the splitting to obtain
the desired extension.  Finally, given this resulting extension, we simply
compute the trace in the usual commutative algebraic geometry sense.
If we were given a splitting of $\bN'/\bE \bN'$, we could instead take the kernel
of the resulting map $\bZ/\bE \bZ\to \bM/\bE \bM$; a splitting of both makes $\bM/\bE \bM$ a
direct summand.

\subsection{}

At this point, we can understand the Poisson structure entirely in terms of
extensions of $\bM$ by $\bM$ together with commutative data; to proceed
further, we will need a more explicit description of self-extensions of
$\bM$.  Suppose that $\bM$ is given by a presentation
\[
0\to V\xrightarrow{L} W\to \bM\to 0,
\]
and consider an extension $0\to \bM\to \bN\to \bM\to 0$.

We first note that if $\Ext^2(W,V)=0$, then there exists a commutative diagram
\[
\begin{CD}
@.0 @.0@.0\\
@. @VVV @VVV @VVV\\
0@>>> V @>L>> W @>>> \bM@>>> 0\\
@. @VVV @VVV @VVV\\
0@>>> V' @>>> W' @>>> \bN@>>> 0\\
@. @VVV @VVV @VVV\\
0@>>> V @>L>> W @>>> \bM@>>> 0\\
@. @VVV @VVV @VVV\\
@.0 @.0@.0
\end{CD}
\]
with short exact rows and columns.
Indeed, we may pull $\bN$ back to an extension of $W$ by $\bM$, which is in the
kernel of the connecting map $\Ext^1(W,\bM)\to \Ext^2(W,V)=0$, and thus is
the pushforward of an extension $W'$, giving a surjective map of short
exact sequences, the kernel of which is as required.

If we further have $\Ext^1(V,V)=\Ext^1(W,W)=0$, then both $V'$ and $W'$ are
trivial extensions, and we find that $\bN$ has a presentation
\[
0\to V\oplus V\xrightarrow{\begin{pmatrix} L & L'\\0 & L\end{pmatrix}}
     W\oplus W\to \bN\to 0.
\]
(This corresponds to the deformation $\Coker(L+\epsilon L')$ over
$\C[\epsilon]/\epsilon^2$.)

With this in mind, we assume
\[
\Ext^2(W,V)=\Ext^1(V,V)=\Ext^1(W,W)=0,
\]
so that extensions of $\bM$ by $\bM$ are represented by maps $L':V\to W$.
(Of course, this representation is by no means unique!)  Given two such
extensions, it is trivial to construct the desired filtered sheaf: $\bZ$ is
simply the kernel of the morphism
\[
\begin{pmatrix}
 L & L' &  0\\
 0 & L  & L''\\
 0 & 0  &  L
\end{pmatrix}
:V^3\to W^3.
\]
(We could equally well take the $13$ entry to be an arbitrary map
$L''':V\to W$; this corresponds to the fact that the class in
$\Ext^1(\bM,\bM/\bE \bM)$ we obtain is only determined modulo the image of
$\Ext^1(\bM,\bM)$.)

If we further assume that $\Tor_1(\bB,W)=0$, so $\Tor_1(\bB,V)=0$ (and recall
we have already assumed $\Tor_1(\bB,\bM)=0$), then we have an exact sequence
\[
0\to \Hom(V,W(-3))\to \Hom(V,W)\to \Hom(V,W/\bE W),
\]
and $\Hom(V,W(-3))\cong \Ext^2(W,V)^*=0$, and thus the extension $L'$ only
depends on its restriction to $\Hom(V,W/\bE W)\cong \Hom_\bB(V/\bE V,W/\bE W)$.  (Note
that if we also assumed $\Ext^1(W,V)=0$, every map in $\Hom_\bB(V/\bE V,W/\bE W)$
would come from a deformation, but we will not need this assumption.)

We thus obtain the following, purely commutative construction.  Given
sheaves (which for our purposes will always be locally free) $V_E$, $W_E$
on $E$ and an injective morphism $L_E:V_E\to W_E$, say that $L'_E:V_E\to
W_E$ is isotrivial if the corresponding deformation of the cokernel is
trivial, or in other words if the extension
\[
\Coker\begin{pmatrix} L_E & L'_E\\0 &L_E\end{pmatrix}
\]
of $\Coker(L_E)$ by $\Coker(L_E)$ splits.  Then we may define a bilinear
form on the space of isotrivial morphisms (or between the space of
isotrivial morphisms and the space of all morphisms) by combining the two
morphisms to a triangular matrix
\[
\begin{pmatrix} L_E & L'_E & 0\\ 0 & L_E & L''_E\\0 & 0 & L\end{pmatrix},
\]
splitting off $\Coker(L_E)$ as a direct summand of the cokernel, then
taking the trace of the class of the corresponding extension.

It turns out this is already enough to let us prove Poissonness in several
important cases.  Suppose, for instance, that $V\cong \bA^n$, $W\cong
\bA[1]^m$; this implies the various vanishing statements we require.  Then
$V_E\cong {\cal O}_E^n$ is independent of $\tau$, while $W_E\cong {\cal
  L}^m$ for a degree $3$ line bundle ${\cal L}$; the latter depends on
$\tau$, but any two such bundles are related under pulling back through a
translation of $E$.  Moreover, a given map $L_E:V_E\to W_E$ lifts to a
unique morphism $L:V\to W$, and $L$ is injective iff $L_E$ is injective.
(Even the condition that $\Coker(L)$ is simple turns out to be reducible to
a question on $L_E$, but in any case, this is an open condition.)  In
particular, given any value of $\tau$, we have an open subspace of the
moduli space parametrizing sheaves with such a presentation, and for any
other value $\tau'$, the corresponding open subspace is birational in a way
preserving the Poisson structure.  In particular, we may take
$\tau'=1_E$, at which point the corresponding moduli space is just a moduli
space of sheaves on $\bP^2$.  Since the Jacobi identity is known to hold
there, it holds on an open subspace for any $\tau$, and thus (since the
failure of the Jacobi identity is measured by a morphism $\wedge^3\Omega\to
{\cal O}$) on the closure of that open subspace, so for any sheaf with a
presentation of the given form.

In fact, with a bit more work, we can extend Poissonness to {\em any}
simple sheaf (apart from point sheaves).  The point is that if $\bM(d)$ is
acyclic for $d\ge -3$, then $\bM$ has a resolution
\[
0\to \bA(-2)^{n_2}\to \bA(-1)^{n_1}\to \bA^{n_0}\to \bM\to 0,
\]
and, as in \cite{Hurtubise/Markman}, we can recover $\bM$ from the cokernel
of the map $\bA(-2)^{n_2}\to \bA(-1)^{n_1}$.  The Poisson structure satisfies
the Jacobi identity in the neighborhood of the latter sheaf (since this is
just a twist of the kind of presentation we have already considered), and
the calculation of \cite{Hurtubise/Markman} shows that the map from a
neighborhood of $\bM$ to this neighborhood simply negates the Poisson
structure.

Note that it follows from this construction that we do not obtain any new
symplectic varieties; every symplectic leaf in the noncommutative setting
is mapped in this way to an open subset of a symplectic leaf in the moduli
space of vector bundles on $\bP^2$.

\subsection{}

The above argument is somewhat unsatisfactory, as it depends on a somewhat
delicate reduction to the commutative case, so is likely to be difficult to
generalize to other noncommutative surfaces (e.g., deformations of del
Pezzo surfaces).  We thus continue our investigation of the pairing.

Since we are now in a completely commutative setting, we may use \v{C}ech
cocycles to perform computations.  In particular, a splitting of the
extension $N'_E$ corresponding to $L'_E$ is a cocycle for $\Hom(N'_E,M_E)$,
while the filtered sheaf $Z_E$ is represented by a cocycle for
$\Ext^1(M_E,N'_E)$.  The desired trace is then simply the trace pairing of
these two classes, which reduces to the trace pairing on matrices.

By the structure of $Z_E$, we find that the cocycle representing $Z_E$ is
simply the (global) morphism
\[
\begin{pmatrix} L''_E& 0 \end{pmatrix} \in \Hom(V_E,W_E^2).
\]
The splitting of $N'_E$ is slightly more complicated.  If we write
$E=U_1\cup U_2$ with $U_1$, $U_2$ affine opens, then the relevant map
$N'_E\to M_E$ is represented over $U_i$ by
\[
\begin{pmatrix}
B'_i\\
0
\end{pmatrix}
\in
\Hom_{U_i}(W_E^2,W_E),
\qquad
\begin{pmatrix}
A'_i\\
0
\end{pmatrix}
\in
\Hom_{U_i}(V_E^2,V_E),
\]
such that
\[
L'_E = B'_i L_E - L_E A'_i,
\]
and there exists
\[
\begin{pmatrix}
\Phi'_{12}\\
0
\end{pmatrix}
\in \Hom_{U_1\cap U_2}(W_E^2,V_E)
\]
such that
\[
B'_2-B'_1 = L_E\Phi'_{12},
\qquad
A'_2-A'_1 = \Phi'_{12}L_E.
\]
Note that since $L_E$ is assumed injective, $\Phi'_{12}$ is uniquely
determined.  Combining this, we find that the trace pairing is given by
\[
-\Tr(L''_E \Phi'_{12})\in \Gamma(U_1\cap U_2,{\cal O}_E),
\]
viewed as a cocycle for $H^1({\cal O}_E)$.

Essentially the same formula (possibly up to sign) appeared in
\cite{Polishchuk}, in which Polishchuk constructed a Poisson structure on
the moduli space of {\em stable} morphisms between vector bundles on $E$.
Although Polishchuk allows the vector bundles to vary, it is easy to check
that any deformation in the image of the cotangent space induces the
trivial deformation of the two bundles.  As a result, Polishchuk's proof of
the Jacobi identity carries over to our case.  (Note that Polishchuk
imposes a stability condition, which is typically stronger than the natural
stability condition in $\Tails A$.  However, all Polishchuk really uses is
that $\Hom(W,V)=0$ and that the complex has no nonscalar automorphisms;
i.e., the natural analogue of ``simple''.)  Note that the interpretation of
Polishchuk's Poisson structure coming from our calculation makes it
straightforward to identify the symplectic leaves: each symplectic leaf
classifies the ways of representing a particular sheaf as the cokernel of a
map $V\to W$ with $V$, $W$ fixed.

In the $1$-dimensional case, the bundles $V_E$, $W_E$ have the same rank,
and thus $L_E$ is generically invertible.  If we choose $U_1$ such that
$L_E$ is invertible on $U_1$, then we can arrange that
\[
A'_1 = L_E^{-1} L'_E, B'_1 = 0,
\]
at which point
\[
\Phi'_{12} = L_E^{-1} B'_2,
\]
so the pairing is given by the cocycle
\[
-\Tr(L''_E L_E^{-1} B'_2).
\]
Given a holomorphic differential $\omega$, the corresponding map to $\C$ is
given by
\[
\sum_{x\in U_2}
\Res_x \Tr(L''_E L_E^{-1} B'_2)\omega.
\]
The contributions come only from those points where $L_E$ fails to be
invertible, i.e., from the support of $M_E$.  Moreover, we readily see that
the local contribution at $x$ will not change if we replace $(A'_2,B'_2)$
by any other splitting holomorphic at $x$.


\begin{thebibliography}{99}

\bibitem{Altman/Kleiman}
A.~B.~Altman and S.~L.~Kleiman,
\emph{Compactifying the Picard scheme},
Adv. in Math., 35(1):50--112, 1980.

\bibitem{Arinkin/Borodin}
D.~Arinkin and A.~Borodin, 
\emph{Moduli spaces of d-connections and difference Painlev\'e
  equations},
 Duke Math.\ J.\ \textbf{134} (2006), no.~3, 515-–556.

\bibitem{Arinkin/Borodin2}
D.~Arinkin and A.~Borodin,
\emph{$\tau$-function of discrete isomonodromy transformations and
  probability}. Compos. Math., 145(3):747--772, 2009.

\bibitem{ArtinSome}
M.~Artin, 
\emph{Some problems on three-dimensional graded domains},
Representation theory and algebraic geometry (Waltham, MA, 1995), 
1–-19, London Math.\ Soc.\ Lecture Note Ser., 
\textbf{238}, Cambridge Univ.\ Press, 1997. 


\bibitem{ATV1}
M.~Artin, J.~Tate, M.~Van den Bergh, 
\emph{
Some algebras associated to automorphisms of elliptic curves}
The Grothendieck Festschrift, Vol.~I, 33--85,
Progr. Math., 86, Birkh\"auser Boston, Boston, MA, 1990. 

\bibitem{ATV2}
M.~Artin, J.~Tate, M.~Van den Bergh, 
\emph{
Modules over regular algebras of dimension $3$}, 
Invent.\ Math.\ \textbf{106} (1991), no.\ 2, 335--388. 

\bibitem{BeauvilleK3}
A.~Beauville,
\emph{Syst\`emes Hamiltoniens compl\'etement int\`egrables associ\'es 
aux surfaces K3},
Problems in the theory of surfaces and their classification (Cortona,
1988), 25–-31
 Sympos.\ Math.\ XXXII, Academic Press, 1991. 

\bibitem{BeauDet}
A.~Beauville, 
\emph{
Determinantal hypersurfaces}, Mich.\ Math.\ J.\ \textbf{48}, 39-64 (2000).

\bibitem{Bottacin}
F.~Bottacin,
\emph{Poisson structures on moduli spaces of sheaves over {P}oisson
  surfaces},
\newblock  Invent.\ Math., 121(2):421--436, 1995.

\bibitem{Hurtubise/Markman}
J.~C. Hurtubise and E.~Markman,
\emph{Elliptic {S}klyanin integrable systems for arbitrary reductive
  groups},
\newblock  Adv.\ Theor.\ Math.\ Phys., 6(5):873--978 (2003), 2002.

\bibitem{HL}
D.~Huybrechts and M.~Lehn,
\emph{The geometry of moduli spaces of sheaves},
Cambridge University Press, 2010.

\bibitem{KMNOY}
K.~Kajiwara, T.~Masuda, M.~Noumi, Y.~Ohta, Y.~Yamada,
\emph{Point configurations, Cremona transformations and the elliptic
  difference Painlev\'e equation}
Th\'eories asymptotiques et \'equations de Painlev\'e, 169--198,
S\'emin Congr., 14, Soc. Mat. France, Paris, 2006.

\bibitem{Kenyon}
R.~Kenyon, 
\emph{Lectures on dimers}, 
Statistical mechanics, 191–-230, 
IAS/Park City Math.\ Ser., \textbf{16}, AMS, 2009. 


\bibitem{Mitchell}
B.~Mitchell,
\emph{Theory of categories},
Pure and Applied Mathematics, Vol. XVII, Academic Press, 1965

\bibitem{NvdB}
K.~de Naeghel and M.~van den Bergh,
\emph{Ideal classes of three-dimensional Sklyanin algebras},
J.\ Algebra \textbf{276} (2004), no.~2, 515–-551. 

\bibitem{Nevins/Stafford}
T.~A. Nevins and J.~T. Stafford.
\emph{Sklyanin algebras and Hilbert schemes of points},
Adv.\ Math.\ 210 (2007), no.~2, 405--478. 

\bibitem{Okounkov1}
A.~Okounkov, 
\emph{Noncommutative geometry of random surfaces}, 
\texttt{arXiv:0907.2322}. 

\bibitem{Okounkov2}
A.~Okounkov, 
\emph{Applied Noncommutative Geometry}, 
2010 Simons lectures at MIT, 
\texttt{http://www.math.columbia.edu/$\sim$okounkov/Simons1.pdf},
\texttt{Simons2.pdf}, \texttt{Simons3.pdf}.



\bibitem{Painleve100}
\emph{The Painlev\'e property. 
One century later.}, ed.\ by Robert Conte,
CRM Series in Mathematical Physics, Springer, 1999.

\bibitem{Polishchuk}
A.~Polishchuk,
\newblock 
\emph{Poisson structures and birational morphisms associated with
bundles on elliptic curves},
\newblock Internat.\ Math.\ Res.\ Notices  (1998), no.~13, 683--703.

\bibitem{Rains:poisson}
E.~M.~Rains,
\newblock \emph{Birational morphisms and Poisson moduli spaces},
\newblock \texttt{arXiv:1307.4032}.

\bibitem{Rains:hitchin}
E.~M.~Rains,
\newblock \emph{Generalized Hitchin systems of rational surfaces},
\newblock \texttt{arXiv:1307.4033}.

\bibitem{Rains2}
E.~M.~Rains, in preparation. 

\bibitem{Sakai}
H.~Sakai, 
\emph{Rational surfaces associated with affine root systems and
  geometry of the Painlev\'e equations}, 
Comm.\ Math.\ Phys.\ \textbf{220} (2001), no.~1, 165–-229. 



\bibitem{SvdB}
J.~T.~Stafford and M.~van den Bergh,
\emph{Noncommutative curves and noncommutative surfaces},
 Bull.\ AMS \textbf{38} (2001), no.~2, 171–-216.

\bibitem{Tyurin}
A.~N. Tyurin.
\newblock 
\emph{Symplectic structures on the moduli spaces of vector bundles on
  algebraic surfaces with {$p_g>0$}}, 
\newblock  Izv.\ Akad.\ Nauk SSSR Ser.\ Mat., 52(4):813--852, 1988.


\end{thebibliography}
\end{document}